\documentclass[a4paper,twoside]{amsart}
\usepackage{amsmath, amssymb}
\usepackage{amsfonts}
\usepackage{hyperref}
\usepackage{a4wide}
\usepackage{graphicx}
\usepackage{float}
\usepackage{srcltx}
\usepackage[all]{xy}
\SelectTips{lu}{12}
\usepackage{cancel}
\usepackage{version}
\usepackage[T1]{fontenc}
\usepackage{xcolor}
\usepackage{enumerate}
\usepackage[inline]{enumitem}
\usepackage[normalem]{ulem}
\usepackage{latexsym}
\usepackage[2emode]{psfrag}
\usepackage{yhmath}
\usepackage{array}
\usepackage{dsfont}
\usepackage{nicefrac}
\usepackage{stmaryrd}
\usepackage[capitalise,noabbrev]{cleveref} 	
\usepackage[final]{showlabels}

\usepackage{mathtools}

\usepackage{etoolbox}
\makeatletter
\patchcmd{\@setaddresses}{\indent}{\noindent}{}{}
\patchcmd{\@setaddresses}{\indent}{\noindent}{}{}
\patchcmd{\@setaddresses}{\indent}{\noindent}{}{}
\patchcmd{\@setaddresses}{\indent}{\noindent}{}{}
\makeatother

\crefname{part}{\S}{\S\S}
\crefname{chapter}{\S}{\S\S}
\crefname{section}{\S}{\S\S}
\crefname{subsection}{\S}{\S\S}

\newtheorem{theorem}{\sc Theorem}[subsection]
\newtheorem{proposition}[theorem]{\sc Proposition}
\newtheorem{notation}[theorem]{\sc Notation}

\newtheorem{lemma}[theorem]{\sc Lemma}
\newtheorem{corollary}[theorem]{\sc Corollary}
\theoremstyle{definition}
\newtheorem{definition}[theorem]{\sc Definition}

\newtheorem{example}[theorem]{\sc Example}

\theoremstyle{remark}
\newtheorem{remark}[theorem]{\sc Remark}

\newtheorem*{proposition*}{Proposition}

\allowdisplaybreaks
\newenvironment{invisible}{{\noindent\sc \colorbox{yellow}{Invisible:}\;}\color{gray}}{\medskip}
\excludeversion{invisible}
\excludeversion{proof?}
\newcommand{\rd}[1]{{\color{red}{#1}}}

\newcommand{\ps}[1]{{\color{blue}{#1}}}
\def\pulb{\ar@{}[dr]|(0.2){\mbox{\Large{$\lrcorner$}}}}

\newcommand{\Mm}{\mathcal{M}}

\newcommand{\id}{\mathrm{Id}}
\newcommand{\rat}{\mathrm{rat}}
\newcommand{\op}{\mathrm{op}}
\newcommand{\cop}{\mathrm{cop}}
\newcommand{\ev}{\mathrm{ev}}

\newcommand{\R}{\mathbb R}

\newcommand{\pred}[1]{{^\star #1}}

\DeclareMathOperator{\conv}{\raisebox{1pt}{$\scriptstyle\star$}}
\DeclareMathOperator{\bra}{\raisebox{+1pt}{$\scriptstyle \blacktriangleleft$}} 
\DeclareMathOperator{\bla}{\raisebox{+1pt}{$\scriptstyle \blacktriangleright$}} 

\begin{document}
\allowdisplaybreaks

\title[Integrals for bialgebras]{Integrals for Bialgebras}

\thanks{This paper was written while the authors were members of the ``National Group for Algebraic and Geometric Structures and their Applications'' (GNSAGA-INdAM).
Esta publicación ha sido financiada con fondos propios de la Junta de Andalucía, en el marco de la ayuda DGP\_EMEC\_2023\_00216.
It is partially based upon work from COST Action CaLISTA CA21109 supported by COST (European Cooperation in Science and Technology) \href{www.cost.eu}{www.cost.eu}.
The authors acknowledge partial support by the European Union -NextGenerationEU under NRRP, Mission 4 Component 2 CUP D53D23005960006 - Call PRIN 2022 No.\, 104 of February 2, 2022 of Italian Ministry of University and Research; Project 2022S97PMY \textit{Structures for Quivers, Algebras and Representations (SQUARE)}}

\begin{abstract}
A well-known result by Larson and Sweedler shows that integrals on a Hopf algebra can be obtained by applying the Structure Theorem for Hopf modules to the rational part of its linear dual.
This fact can be rephrased by saying that taking the space of integrals comes from a right adjoint functor from a category of modules to the category of vector spaces.
This observation inspired the categorical approach that we advocate in this work, which yields to a new notion of integrals for bialgebras in the linear setting. Despite the novelty of the construction, it returns the classical definition in the presence of an antipode.
%
%
%
We test this new concept on bialgebras that satisfy at least one of the following properties: being coseparable as regular module coalgebras, having a one-sided antipode, being commutative, being cocommutative, or being finite-dimensional.
One of the main results we obtain in this process is a dual Maschke-type theorem relating coseparability and total integrals.
Remarkably, there are cases in which the space of integrals turns out to be isomorphic to that of the associated Hopf envelope.
In particular, this space results to be one-dimensional for finite-dimensional bialgebras, providing an existence and uniqueness theorem for integrals in the finite-dimensional case.
Furthermore, explicit computations are given for concrete examples including the polynomial bialgebra with one group-like variable, the quantum plane and the coordinate bialgebra of $n$-by-$n$ matrices.

%


\end{abstract}


\author{Alessandro Ardizzoni}
\address{%
\parbox[b]{\linewidth}{University of Turin, Department of Mathematics ``G. Peano'', via
Carlo Alberto 10, I-10123 Torino, Italy}}
\email{alessandro.ardizzoni@unito.it}
\urladdr{\url{www.sites.google.com/site/aleardizzonihome}}

\author{Claudia Menini}
\address{%
\parbox[b]{\linewidth}{University of Ferrara, Department of Mathematics and Computer Science, Via Machiavelli
30, Ferrara, I-44121, Italy}}
\email{men@unife.it}
\urladdr{\url{https://sites.google.com/a/unife.it/claudia-menini}}

\author{Paolo Saracco}
\address{Universidad de Sevilla, Departamento de \'Algebra, Facultad de Matemáticas, Avda. Reina Mercedes s/n, Apdo. 1160, 41080, Sevilla, Spain.}
\email{psaracco@us.es}
\urladdr{\url{https://sites.google.com/view/paolo-saracco}}


\maketitle
\tableofcontents

\section{Introduction}


An extremely familiar notion in Hopf algebra theory is that of an \emph{integral}. If $G$ is a compact topological group, then a Haar integral on $G$ is a linear functional $\int \colon \mathcal{C}(G) \to \R$ which is translation invariant, in the sense that $\int (f \leftharpoonup y) = \int (f)$ for all $f \in \mathcal{C}(G)$ and for all $y \in G$. It follows that the restriction of $\int$ to the Hopf $\R$-algebra $\mathcal{R}(G)$ of representative functions on $G$ is an element in $\mathcal{R}(G)^*$ such that
\begin{enumerate*}[label=(\roman*)]
\item $\alpha * \int = \alpha(1)\int$ for every $\alpha \in \mathcal{R}(G)^*$ and
\item $\int (f^2) > 0$ for every $f \neq 0$ in $\mathcal{R}(G)$.
\end{enumerate*}
The Tannaka Duality Theorem then establishes an equivalence of categories between compact topological groups and commutative Hopf $\R$-algebras with a Haar integral $\int$ and whose group of algebra morphisms to $\R$ is dense in the linear dual (see, for instance, \cite[Theorem 3.5]{Hochschild} or its more modern rephrasing \cite[Theorem 3.4.3]{Abe}).
As a consequence, if $H$ is any Hopf algebra over a field $\Bbbk$ and $\lambda \in H^*$ is a linear functional such that
\begin{equation}\label{eq:classint}
\mu * \lambda = \mu(1)\lambda \quad \text{ for all }  \quad \mu \in H^*,
\end{equation}
then $\lambda$ is called a \emph{left integral} for (or on) $H$ and the space they form is denoted by $\int_l H^*$. Hopf algebras which admit such an integral can be considered as noncommutative abstractions of the algebras of functions on compact groups.

Integrals for Hopf algebras are well-studied in the literature and a number of outstanding results were proved since their introduction. The existence and uniqueness (up to scalar multiples) of integrals for finite-dimensional Hopf algebras, for example, was already proved in the seminal paper \cite{LarsonSweedler} of Larson and Sweedler, together with a Maschke-type theorem for Hopf algebras, and then re-established with different approaches in \cite{KauffmanRadford} and \cite{VanDaele}. Integrals for infinite dimensional Hopf algebras were also considered by Sweedler in \cite{Swe-Integ}, together with some corresponding representation theoretic properties for the Hopf algebra and for the category of comodules over it, 
and uniqueness of integrals for them was established by Sullivan in \cite{Sullivan} (see also \cite{Stefan}). The case of finitely generated and projective Hopf algebras over a commutative ring and their connections with Frobenius algebras was tackled by Pareigis in \cite{Pareigis-Frob}.
The subsequently developed theory of integrals turned into a powerful instrument for the study of finite-dimensional Hopf algebras. For instance, Radford used it to prove that the order of the antipode of a finite-dimensional Hopf algebra is finite \cite{Radford-Order} and the celebrated Larson-Sweedler theorem \cite{LarsonSweedler} entails that a finite-dimensional bialgebra is a Hopf algebra if and only if it admits a non-degenerate left integral.
Therefore, it is not surprising that integrals and their study have been introduced and investigated for almost all the Hopf-like structures, from (co)quasi-Hopf algebras \cite{BulacuCaenepeel,HausserNill,FlorinFred}, to braided Hopf algebras \cite{BKLT}, weak Hopf algebras \cite{GabiNillSzlachanyi}, Hopf algebroids \cite{GabiHA}, Hopf monoids \cite{Gabi-HopfMonoids}, and also for some more general notions, such as coalgebras \cite{Iovanov-AbstrInteg}.

In the present paper, we focus on bialgebras over a field and we offer a novel perspective on integrals which differs profoundly from the more classical one, where they are defined as those linear functionals on $B$ satisfying \eqref{eq:classint}. This relies on a broad categorical framework, which
leads to a fresh new notion of integrals for bialgebras. Naively speaking, these can interpreted as translation-invariant linear functionals on the associated quantum monoid
and, differently from the classical ones, they carry more significant information about the bialgebra itself. For example, we will show in a dual Maschke-type theorem how the existence of total integrals in our sense relates to coseparability (see \cref{prop:cosep} and \cref{cor:totiBinj}). Furthermore, these integrals behave more reasonably than their classical counterpart, in the sense that we can prove an existence and uniqueness theorem for integrals on finite-dimensional, and even left Artinian, bialgebras, see \cref{coro:iBcopsu}. The fact that this does not hold, in general, for the classical integrals is highlighted in \cref{ex:findimcase} and \cref{ex:Atimesk}.
%

It is noteworthy that the existence of these new integrals
does not require, and does not imply, the existence of an antipode. In fact, over Hopf algebras they return the usual ones from the literature.
In addition, they often offer information on the integrals for the Hopf envelope associated with the given bialgebra, since in a number of cases of interest one space of integrals naturally identifies with the other. In this, a key role is played by the $\oslash$-construction from \cite{ArMeSa}. Indeed, integrals for a bialgebra $B$ are distinguished linear functionals in its pre-dual $\pred{B} = (B \oslash B)^*$ and for certain bialgebras, such as the left Artinian or the commutative ones, it turns out that the Hopf envelope can be realised exactly as $B \oslash B = (B \otimes B)/(B \otimes B) B^+$, emphasizing a neat connection between the Hopf envelope on the one hand, and our integrals on the other.

All in all, this novel concept is more flexible and paves the way for more general applications, while still returning the classical definition in the presence of an antipode.

Concretely, we begin by revising the classical case \cref{ssec:classical}, in order to clarify, support and justify the categorical approach to be followed in \cref{ssec:cats}. Then, we interpret our general construction in the specific case of bialgebras in \cref{sec:bialg}, which leads to the main definition of our work: \cref{def:oslash}. Discussing some concrete realisations of it took us to realise the distinguished role that a certain algebraic object plays in this situation: the canonical morphism $i_B$ from \cite{ArMeSa}, which appears in approaching the celebrated structure theorem for Hopf modules (see \cref{ssec:iB}). In fact, properties of $i_B$, such as its injectivity or surjectivity, simplifies the study of our integrals and allow us to relate them with the traditional ones (see, for instance, \cref{ssec:cocomm}).
After concluding this general introduction, i.e.\ \cref{sec:general}, we move forward to analyse some cases of interest in \cref{sec:partcases}, such as the connections with coseparability \cref{ssec:cosep}, the integrals for one-sided Hopf algebras \cref{ssec:intnHopf}, and the  cocommutative \cref{ssec:cocomm}, commutative \cref{ssec:commcase}, and finite-dimensional \cref{ssec:findim} cases.

In particular, with \cref{prop:cosep} and \cref{cor:totiBinj} we provide an analogue of the dual Maschke theorem (see \cite[Theorem]{Larson}) for bialgebras: a bialgebra $B$ which is coseparable as a $B$-module coalgebra, always admits a total integral. Conversely, the existence of a total integral implies coseparability whenever the canonical map $i_B$ is injective.

We also describe the space of integrals on a right Hopf algebra $B$ as a suitable subspace of the linear dual $B^*$ in \cref{pro:intgRHopf}.


Moreover, we prove a number of uniqueness results for integrals for cocommutative bialgebras. On the one hand, we show that the space of integrals for bialgebras admitting a total integral is one-dimensional; see \cref{lem:tot}. On the other hand, we naturally identify the space of integrals for a bialgebra $B$ with the space of integrals on the Hopf envelope $\mathrm{H}(B)$, provided the canonical map $i_B$ is surjective; see \cref{prop:intgccom-iBsu}.
Similar results are obtained for commutative and finite-dimensional bialgebras, or more generally, for left Artinian bialgebras, even without conditions on $i_B$. Furthermore, in the finite-dimensional case, we show in \cref{coro:iBcopsu} that the space of integrals is one-dimensional, analogous to the behaviour of integrals for Hopf algebras.

We conclude with \cref{sec:examples}, where we compute explicitly and in details the integrals in a number of concrete examples such as the polynomial bialgebra with one group-like variable, the quantum plane and the coordinate bialgebra of $n$-by-$n$ matrices.
Remarkably, in all these cases the space of integrals of the relevant bialgebra $B$ identifies with that of its Hopf envelope $\mathrm{H}(B)$ and is therefore at most one-dimensional.

\section{Integrals for bialgebras}\label{sec:general}

Throughout the paper, $\Bbbk$ denotes a field of arbitrary characteristic and not necessarily algebraically closed and  $\mathfrak{M}$ is the category of vector spaces over $\Bbbk$. As a matter of notation, the unadorned tensor product $\mbox{-}\otimes\mbox{-}$ denotes the tensor product over $\Bbbk$ of vector spaces and $\mathfrak{M}(\mbox{-},\mbox{-})$ denotes the vector space of $\Bbbk$-linear maps. More generally, if $\mathcal{C}$ is any locally small category and $X,Y$ are two objects in $\mathcal{C}$, then $\mathcal{C}(X,Y)$ denotes the set of morphisms in $\mathcal{C}$ from $X$ to $Y$. Exceptionally, we may write $\mathrm{End}_\Bbbk(\mbox{-})$ for the vector space of $\Bbbk$-linear endomorphisms.

\subsection{The classical case revised}\label{ssec:classical}
To begin with, we revisit the classical case as a prelude to the more general approach.
If $H$ is a Hopf algebra over $\Bbbk$ with antipode $S$, then we can consider its linear dual  $H^{\ast} \coloneqq \mathfrak{M}(H,\Bbbk)$.
It is well-known that the rational part $H^{\ast \rat}$ of $H^\ast$ becomes an object in the category $\mathfrak{M}_{H} ^{H}$ of (right) $H$-Hopf modules, see \cite[Lemma 2.8]{Swe-Integ}, and that the right $H$-coinvariant elements in $H^{\ast \rat}$ are exactly the left integrals on $H$, i.e.~linear maps $\tau:H\to\Bbbk$ such that $\sum h_1\tau(h_2)=1_H\tau(h)$ for every $h\in H$, see \cite[page 330]{Swe-Integ}.\medskip

Our initial goal is to establish \cref{pro:Rclassic}, which will enable us to derive a functorial interpretation of the $H$-Hopf module structure on $H^{\ast \rat}$.
First note that $H^*$ is an algebra through the convolution product  $\ast $. As usual, define $(h\rightharpoonup f)(x) \coloneqq f(xh)$ and
$\left(
f\leftharpoondown h\right) \left( x\right) \coloneqq f\left( xS\left( h\right)
\right)$ for every $f\in H^*,x,h\in H$, so that $f\leftharpoondown h=S(h)\rightharpoonup f$. One easily checks that
\[
\left( f\ast g\right)\leftharpoondown h=\sum \left( f\leftharpoondown
h_{2}\right) \ast \left( g\leftharpoondown h_{1}\right) \qquad\text{and}\qquad
\varepsilon \leftharpoondown h=\varepsilon \left( h\right) \varepsilon .
\]
This means that $\left( H^{\ast },\leftharpoondown \right) $ is a right $H^{%
\cop }$-module algebra, i.e.~an algebra in the monoidal category $%
\left( \mathfrak{M}_{H^{\cop }},\otimes ,\Bbbk \right) .$
Thus, it makes sense to consider the category
$_{H^\ast }\left(\mathfrak{M}_{H^\cop}\right)$ of left $H^*$-modules in $\mathfrak{M}_{H^\cop}$.

\begin{proposition}\label{pro:Rclassic}
Let $H$ be a Hopf algebra.
Then the usual adjunction $L\dashv R:{}_{H^*}\mathfrak{M}\to \mathfrak{M}^H$ that yields the isomorphism between the category of right $H$-comodules and the category of rational left $H^*$-modules, induces an adjuction $L\dashv R:{_{H^{\ast }}\left(
\mathfrak{M}_{H^{\cop }}\right) }\to \mathfrak{M}_{H} ^{H}$
as follows.

The left adjoint is the functor%
\begin{equation*}
L: \mathfrak{M}_{H} ^{H}\rightarrow {_{H^{\ast }}\left(
\mathfrak{M}_{H^{\cop }}\right) },\qquad \left( M,\mu^r,\rho \right)\mapsto \left( (M,\mu^r),\mu^l _{\rho }\right)\qquad f\mapsto f
\end{equation*}%
 where $%
\mu^l_{\rho }\left( f\otimes m\right) \coloneqq \sum m_{0}f\left( m_{1}\right) $.
Moreover  the functor $L$ is fully faithful and its right adjoint is
\begin{equation*}
R:{_{H^{\ast }}\left( \mathfrak{M}_{H^{\cop }}\right) }\rightarrow
\mathfrak{M}_H^H,\qquad ((N,\mu^r),\mu^l)\mapsto(N^{\mathrm{rat}},\mu_\rat^r, \rho_\rat),\qquad f\mapsto f^{\mathrm{rat}},
\end{equation*}%
where $N^{\mathrm{rat}}$ is the rational part of the left $H^{\ast }$%
-module $\left( N,\mu^l \right) $,
the map
$\mu^r_\rat: N^\rat\otimes H\to N^\rat$ is induced by $\mu^r:N\otimes H\to N$
and the map
$\rho_\rat:N^\rat\to N^\rat\otimes H$ is uniquely defined by $\mu^l_\rat=\mu^l_{\rho_\rat}$ where $\mu^l_\rat:H^*\otimes N^{\mathrm{rat}}\to N^{\mathrm{rat}}$ is induced by $\mu^l:H^*\otimes N\to N$.
\end{proposition}

\begin{proof}
Let $(M,\mu^r,\rho )$ be a $H$-Hopf module. Then $(M,\mu^r)$ is an object in $\mathfrak{M}_{H^{\cop }}$.
We check that $\mu^l_\rho:H^*\otimes M\to M$ is a morphism in $\mathfrak{M}_{H^{\cop }}$:
\begin{align*}
\mu^l _{\rho }\left( \left( f\otimes m\right) h\right) & = \mu^l _{\rho }\left( \left( f\leftharpoondown h_{2}\right)\otimes mh_{1}
\right) =\sum \left( mh_{1}\right) _{0}\left( f\leftharpoondown h_{2}\right)
\left( \left( mh_{1}\right) _{1}\right) \\
& = \sum m_{0}h_{1}\left( f\leftharpoondown h_{3}\right) \left(
m_{1}h_{2}\right) =\sum m_{0}h_{1}f\left( m_{1}h_{2}S\left( h_{3}\right)
\right)  \\
& = \sum m_{0}hf\left( m_{1}\right) = \sum m_{0}f\left( m_{1}\right) h=\mu^l _{\rho }\left( f\otimes m\right) h
\end{align*}%
Thus, since $\left( M,\mu^l _{\rho }\right) $ is a left $H^{\ast }$%
-module in $\mathfrak{M}$, we get that $\left(( M,\mu^r),\mu^l _{\rho }\right) $ is a left $H^{\ast }$%
-module in $\mathfrak{M}_{H^{\cop }}$.

Let now $((N,\mu^r),\mu^l)$ be an object in ${_{H^{\ast }}\left( \mathfrak{M}_{H^{\cop }}\right) }$.
We know that $(N^{\mathrm{rat}},\rho_\rat)$ is a right $%
H$-comodule.
Let us check that $N^\rat$ is a right  $H$-submodule of $(N,\mu^r)$
 and hence
$\mu^r$ induces an action $\mu^r_\rat:N^\rat\otimes H\to N^\rat.$
For $%
n\in N^{\mathrm{rat}},f\in H^{\ast },$ we compute%
\begin{align*}
f\cdot \left( nh\right) &=\mu^l \left( f\otimes nh\right) =\mu^l \left( S(
h_{2}) h_{3}\rightharpoonup f\otimes nh_{1}\right) =\mu^l \left( \left(
h_{3}\rightharpoonup f\right) \leftharpoondown h_{2}\otimes nh_{1}\right) \\
& = \mu^l \left( \mu^r\left(\left(
h_{2}\rightharpoonup f \otimes n\right) \otimes h_{1}\right)\right) = \mu^l \left( \left( h_{2}\rightharpoonup f\right) \otimes n\right)
h_{1}\\
& = \left( \left( h_{2}\rightharpoonup f\right) \cdot n\right) h_{1}=\sum
\left( n_{0}\left( h_{2}\rightharpoonup f\right) \left( n_{1}\right) \right)
h_{1}
=\sum n_{0}h_{1}f\left( n_{1}h_{2}\right) .
\end{align*}%
This proves that $nh\in N^{\mathrm{rat}}$, so that $N^{\mathrm{rat}}$ is a right $H^\cop$-submodule of $N$, and that $\rho_\rat \left( nh\right)
=\sum n_{0}h_{1}\otimes n_{1}h_{2}$, so that $(N^\rat,\mu^r_\rat,\rho_\rat)$ is a $H$-Hopf module as desired.\\
It is now easy to see that $\left( L,R\right) $ is an adjunction with $L$ fully faithful (the unit is the identity).
\begin{invisible}
Given an object $((N,\mu^r),\mu^l)$  in ${_{H^{\ast }}\left( \mathfrak{M}_{H^{\cop }}\right) }$ and an object $(M,\mu^r,\rho )$ in $\mathfrak{M}_H^H$, we have
\begin{align*}
  LR((N,\mu^r),\mu^l) & =L(N^\rat,\mu^r_\rat, \rho_\rat)=((N^\rat,\mu^r_\rat),\mu^l_{\rho_\rat})=((N^\rat,\mu^r_\rat),\mu^l_\rat)\\
  RL(M,\mu^r,\rho ) &=R((M,\mu^r),\mu^l_\rho)=(M^\rat,\mu^r_\rat,\rho_\rat)=(M,\mu^r,\rho ).
\end{align*}
Thus we can define the counit of the adjunction $\epsilon((N,\mu^r),\mu^l):((N^\rat,\mu^r_\rat),\mu^l_\rat)\to
((N,\mu^r),\mu^l)$ as the canonical inclusion and take the unit to be $\eta=\id$. Since $((N^\rat,\mu^r_\rat),\mu^l_\rat)$ is a left $H^*$-submodule and a right $H^\cop$-submodule of $((N,\mu^r),\mu^l)$, we get that $\epsilon((N,\mu^r),\mu^l)$ is a morphism in ${}{_{H^{\ast }}\left( \mathfrak{M}_{H^{\cop }}\right) }$. Moreover, since $R$ acts as the restriction on morphisms, $\epsilon((N,\mu^r),\mu^l)$ comes out to be natural and hence to define a natural transformation $\epsilon:LR\to \id$. Moreover $R\epsilon((N,\mu^r),\mu^l):RLR((N,\mu^r),\mu^l)=R((N,\mu^r),\mu^l)\to R((N,\mu^r),\mu^l)$ is just the restriction of $\epsilon((N,\mu^r),\mu^l)$ whence the identity, and $\epsilon L(M,\mu^r,\rho ):LRL(M,\mu^r,\rho )=L(M,\mu^r,\rho )\to L(M,\mu^r,\rho )$ is the canonical inclusion whence the identity as well. Thus $R\epsilon=R$ and $\epsilon L =L$ and hence $(L,R,\eta,\epsilon)$ is an adjunction with $L$ fully faithful (i.e. $R$ is a coreflection) as $\eta$ is invertible.
\end{invisible}
\end{proof}

\begin{remark}
   In the particular case when $H$ is finite-dimensional, every left $H^*$-module is rational and hence the counit $\epsilon$ becomes the identity, too. In this case $(L,R)$ is a category isomorphism.
\end{remark}

We now turn to the desired functorial interpretation of the $H$-Hopf module structure on $H^{\ast \rat}$. Note that $H^{\ast }$ itself is
an object in ${_{H^{\ast }}\left( \mathfrak{M}_{H^{\cop %
}}\right) }$ through its multiplication. Therefore, in view of \cref{pro:Rclassic}, we can compute $RH^{\ast }=H^{\ast \mathrm{rat}}$ which then gains a structure of $H$-Hopf module. This shows how the usual $H$-Hopf module structure of $H^{\ast \mathrm{rat}}$ arises from the existence of a right adjoint $R$ for the functor $L$. We mentioned that left integrals are the coinvariant elements in $H^{\ast \mathrm{rat}}$. It happens that the fully faithful functor $F=(-)\otimes H^\bullet_\bullet:\mathfrak{M}\to \mathfrak{M}^H_H,V\mapsto V\otimes H^\bullet_\bullet$, admits as right adjoint the functor $G=(-)^{\mathrm{co}H}$, which computes the coinvariants. The composition of the adjunctions $(L,R)$ and $(F,G)$ yields the adjunction $(LF,GR)$ where
\begin{equation*}
GR:{_{H^{\ast }}\left( \mathfrak{M}_{H^{\cop }}\right) }\rightarrow
\mathfrak{M},\qquad ((N,\mu^r),\mu^l)\mapsto (N^{\mathrm{rat}})^{\mathrm{co} H},\qquad f\mapsto (f^{\mathrm{rat}})^{\mathrm{co} H}.
\end{equation*}
Thus, the space of left integrals on $H$ can be realised as 
\[\textstyle\int_l H^*=(H^{*\mathrm{rat}})^{\mathrm{co} H}=GR(H^*).\]
Summing up, the following functors played a central role in the foregoing
\begin{equation}\label{diag:func}
\xymatrix{\mathfrak{M}\ar@<.5ex>[r]^-F& \mathfrak{M}_H^H\ar@<.5ex>[l]^--{G}\ar@<.5ex>[r]^-L& {_{H^{\ast }}\left( \mathfrak{M}_{H^{\cop }}\right) }\ar@<.5ex>[l]^-{R} }.
\end{equation}

In the construction above the existence of an antipode $S$ for $H$ is explicitly used to endow $H^*$ with a structure of right $H^\cop$-module and then consider the category ${_{H^{\ast }}\left( \mathfrak{M}_{H^{\cop }}\right) }$.\medskip

In order to get rid of the antipode and pave the way toward possible extensions to other contexts, we need the following easy result (an analogue for right modules over a left $H$-comodule  algebra $K$  is observed in \cite[end of page 362]{Sch02}).

\begin{lemma}
\label{lem:MBA}
Let $B$ be a bialgebra and let $A$ be a right $B$-module algebra. Then we have a category isomorphism $(\mathfrak{M}_B)_A\cong {}_{A^\op}(\mathfrak{M}_ {B^\cop}).$
\end{lemma}
\begin{proof}
We include the proof for the reader's convenience. Denote by $a^\op$ an element of $a\in A$ regarded as an element in $A^\op$. It is well-known that right $A$-modules and left $A^\op$-modules correspond to each other via the assignment $(M,\alpha:M\otimes A\to M)\leftrightarrow (M,\alpha^\op:A^\op\otimes M\to M)$ defined by $\alpha^\op(a^\op\otimes m)=\alpha(m\otimes a)$. On the other hand, right $B$-modules and right $B^\cop$-modules correspond to each other via the assignment $(M,\beta :M\otimes B\to M)\leftrightarrow (M,\beta^\cop :M\otimes B^\cop\to M)$ defined by $\beta^\cop (m\otimes b)=\beta (m\otimes b)$. First note that $A^\op$ becomes an algebra in $\mathfrak{M}_ {B^\cop}$ with action $A^\op\otimes B^\cop\to A^\op,a^\op\otimes b\mapsto a^\op b\coloneqq (ab)^\op$.
Define the candidate category isomorphism $F:(\mathfrak{M}_B)_A\to {}_{A^\op}(\mathfrak{M}_ {B^\cop})$ that takes $(M,\alpha,\beta )$, to $(M,\alpha^\op,\beta^\cop )$. Set $ma\coloneqq \alpha(m\otimes a)$, $mb\coloneqq \beta(m\otimes b)$ and $a^\op m\coloneqq \alpha^\op(a^\op\otimes m)=ma$. Since $(a^\op m)b=(ma)b$ and $\sum(a^\op b_2) (mb_1)=\sum(ab_2)^\op (mb_1)=\sum(mb_1)(ab_2)$, we get that $(a^\op m)b=\sum(a^\op b_2) (mb_1)$ if and only if $(ma)b=\sum(mb_1)(ab_2)$ i.e. $\alpha^\op$ is right $B^\cop$-linear if and only if $\alpha$ is right $B$-linear.
\end{proof}

Set ${^*H}\coloneqq (H^*)^\op$. By means of the previous lemma in case $A={^*H}$ and $B=H$, the  functor $L$ of \cref{pro:Rclassic} can be rephrased as
\begin{equation*}
\xymatrix{\left( \mathfrak{M}_{H}\right) ^{H}=\mathfrak{M}_H^H\ar[r]^-{L}&{_{H^{\ast }}\left(
\mathfrak{M}_{H^{\cop }}\right) }\cong (\mathfrak{M}_H)_{^*H} }
\end{equation*}
Our next aim is to further modify the functor $L:\left( \mathfrak{M}_{H}\right) ^{H}\to (\mathfrak{M}_H)_{^*H}$ above by replacing $^*H$ with another object that becomes an algebra in the monoidal category $(\mathfrak{M}_H,\otimes,\Bbbk)$ with no requirement of an antipode and to prove that such an $L$ has a right adjoint. The reader not interested in the categorical framework behind this construction can directly move to \cref{sec:bialg} where the resulting structure is expounded.
Here we just mention that we are going to replace the category $\mathfrak{M}_H$ by a monoidal category $\Mm$, $H$ by a coalgebra $C$ in $\Mm$, and we are going to look for a functor $L:\mathcal{M} ^{C}\to \mathcal{M}_\pred{C}$ where $\pred{C}$ is a suitable algebra attached to $C$. This is the main subject of the following section.

\subsection{A categorical setting}\label{ssec:cats}

Let $(\mathcal{M},\otimes,\mathbf{1})$ be a monoidal category with constraints $a,l,r$. For the sake of clarity, we often keep track of the unit constraints $l$ and $r$ in computations, but we omit the associativity constraint $a$. Here, given a coalgebra $C$ in $\mathcal{M}$ and under the requirement that $C$ has a left pre-dual $\pred{C}$, we endow $\pred{C}$ with an algebra structure (\cref{lem:*algebra}) and construct  a functor  $L:\mathcal{M} ^{C}\rightarrow \mathcal{M}_{\pred{C}}$ (\cref{pro:L}).

\begin{definition}
By adapting \cite[4.1.3]{Goyvaerts-Vercruysse}, given a monoidal category $%
\mathcal{M}$ we say that an object $X\in \mathcal{M}$ is \emph{%
left pre-rigid} if there is an object $\pred{X}$ (a \emph{left pre-dual}
of $X$) and a morphism $\mathrm{ev}_{X}:X\otimes \pred{X}\rightarrow
\mathbf{1}$ (the \emph{evaluation}) such that the map
\begin{equation*}
\mathcal{M}\left( T,\pred{X}\right) \rightarrow \mathcal{M}\left(
X\otimes T,\mathbf{1}\right) :f\mapsto \mathrm{ev}_{X}\circ \left( X\otimes
f\right),
\end{equation*}%
is bijective for every object $T$ in $\mathcal{M}$. The category $\mathcal{M}$ is said to be \emph{left pre-rigid} if every object in $\mathcal{M}$ is left pre-rigid.
\end{definition}

\begin{example} \label{rmk:closed=>prig}
Assume $X\otimes \left( -\right) :\mathcal{M}\rightarrow \mathcal{M}$ has a
right adjoint $\left[ X,-\right] $ (e.g. the monoidal category $\mathcal{%
M}$ is \emph{left closed}). We set $\pred{X}\coloneqq \left[ X,\mathbf{1}\right].$
In particular, we have the bijection%
\begin{equation*}
\mathcal{M}\left( T,\pred{X}\right) =\mathcal{M}\left( T,\left[ X,\mathbf{%
1}\right] \right) \cong \mathcal{M}\left( X\otimes T,\mathbf{1}\right) .
\end{equation*}%
If we take $T=\pred{X},$ we get the bijection
$\mathcal{M}\left(\pred{X},\pred{X}\right) \cong \mathcal{M}\left( X\otimes \pred{X},\mathbf{1}%
\right) .$ Denote by $\mathrm{ev}_{X}:X\otimes \pred{X}\rightarrow
\mathbf{1}$ the morphism corresponding to $\mathrm{Id}_{\pred{X}}.$ Since $\mathrm{ev}_{X}$ is the counit of the adjunction evaluated at $%
\mathbf{1}$, we get that the starting bijection above is given by the
assignments $f\mapsto \mathrm{ev}_{X}\circ \left( X\otimes f\right) .$
In other words, $X$ is left pre-rigid for every object $X$ and hence $\mathcal{M}$ is left pre-rigid.
\end{example}


\begin{lemma}
\label{lem:*algebra} Let $\mathcal{M}$ be a monoidal category and let $%
\left( C,\Delta _{C},\varepsilon _{C}\right) $ be a coalgebra in $\mathcal{M}
$. Assume that $C$ is a left pre-rigid object with left pre-dual $\pred{C}$. Then
$\left( \pred{C},m_\pred{C},u_\pred{C}\right) $  is an algebra in $\mathcal{M}$ where the multiplication $m_\pred{C}$ and the unit $u_\pred{C}$ are uniquely defined by the equalities%
\begin{align}
\mathrm{ev}_{C}\circ \left( C\otimes m_\pred{C}\right) & =\mathrm{ev}%
_{C}\circ \left( r_{C}{\otimes \pred{C}}\right) \circ \left( C\otimes
\mathrm{ev}_{C}{\otimes \pred{C}}\right) \circ \left( \Delta _{C}\otimes {%
\pred{C}}\otimes \pred{C}\right) ,  \label{form:dualmult} \\
\mathrm{ev}_{C}\circ \left( C{\otimes }u_{\pred{C}}\right) & =\varepsilon
_{C}\circ r_{{C}}.  \label{form:dualun}
\end{align}
\end{lemma}

\begin{proof}
By the bijections above there are unique morphisms $m_\pred{C}$ and $u_{%
\pred{C}}$ defined by the equalities in the statement. Let us check the
associativity. We compute%
\begin{eqnarray*}
&&\mathrm{ev}_{C}\circ \left( C\otimes m_\pred{C}\right) \circ \left(
C\otimes m_\pred{C}\otimes \pred{C}\right) \\
&=&\mathrm{ev}_{C}\circ \left( r_{C}\otimes \pred{C}\right) \circ \left(
C\otimes \mathrm{ev}_{C}\otimes \pred{C}\right) \circ \left( \Delta
_{C}\otimes \pred{C}\otimes \pred{C}\right) \circ \left( C\otimes m_{{%
\pred{C}}}\otimes \pred{C}\right) \\
&=&\mathrm{ev}_{C}\circ \left( r_{C}\otimes \pred{C}\right) \circ \left(
C\otimes \mathrm{ev}_{C}\otimes \pred{C}\right) \circ \left( C\otimes
C\otimes m_\pred{C}\otimes \pred{C}\right) \circ \left( \Delta
_{C}\otimes \pred{C}\otimes \pred{C}\otimes \pred{C}\right) \\
&=&\mathrm{ev}_{C}\circ \left( r_{C}\otimes \pred{C}\right) \circ \left(
C\otimes \mathrm{ev}_{C}\otimes \pred{C}\right) \circ \left( C\otimes
r_{C}{\otimes \pred{C}}\otimes \pred{C}\right) \circ \left( C\otimes
C\otimes \mathrm{ev}_{C}\otimes \pred{C}\otimes \pred{C}\right) \\
&&\circ \left( C\otimes \Delta _{C}\otimes \pred{C}\otimes \pred{C}%
\otimes \pred{C}\right) \circ \left( \Delta _{C}\otimes \pred{C}%
\otimes \pred{C}\otimes \pred{C}\right) \\
&=&\mathrm{ev}_{C}\circ \left( r_{C}\otimes \pred{C}\right) \circ \left(
C\otimes \mathrm{ev}_{C}\otimes \pred{C}\right) \circ \left( C\otimes
r_{C}{\otimes \pred{C}}\otimes \pred{C}\right) \circ \left( C\otimes
C\otimes \mathrm{ev}_{C}\otimes \pred{C}\otimes \pred{C}\right) \\
&&\circ \left( \Delta _{C}\otimes C\otimes \pred{C}\otimes \pred{C}%
\otimes \pred{C}\right) \circ \left( \Delta _{C}\otimes \pred{C}%
\otimes \pred{C}\otimes \pred{C}\right) \\
&=&\mathrm{ev}_{C}\circ \left( r_{C}\otimes \pred{C}\right) \circ \left(
C\otimes \mathrm{ev}_{C}\otimes \pred{C}\right) \circ \left( r_{C\otimes
C}{\otimes \pred{C}}\otimes \pred{C}\right) \circ \left( \Delta
_{C}\otimes \mathbf{1}\otimes \pred{C}\otimes \pred{C}\right) \\
&&\circ \left( C\otimes \mathrm{ev}_{C}\otimes \pred{C}\otimes \pred{C}%
\right) \circ \left( \Delta _{C}\otimes \pred{C}\otimes \pred{C}%
\otimes \pred{C}\right) \\
&=&\mathrm{ev}_{C}\circ \left( r_{C}\otimes \pred{C}\right) \circ \left(
C\otimes \mathrm{ev}_{C}\otimes \pred{C}\right) \circ \left( \Delta
_{C}\otimes \pred{C}\otimes \pred{C}\right) \circ \left( r_{C}{\otimes
\pred{C}}\otimes \pred{C}\right) \\
&&\circ \left( C\otimes \mathrm{ev}_{C}\otimes \pred{C}\otimes \pred{C}%
\right) \circ \left( \Delta _{C}\otimes \pred{C}\otimes \pred{C}%
\otimes \pred{C}\right) \\
&=&\mathrm{ev}_{C}\circ \left( C\otimes m_\pred{C}\right) \circ \left(
r_{C}\otimes \pred{C}\otimes \pred{C}\right) \circ \left( C\otimes
\mathrm{ev}_{C}\otimes \pred{C}\otimes \pred{C}\right) \circ \left(
\Delta _{C}\otimes \pred{C}\otimes \pred{C}\otimes \pred{C}\right)
\\
&=&\mathrm{ev}_{C}\circ \left( r_{C}\otimes \pred{C}\right) \circ \left(
C\otimes \mathrm{ev}_{C}\otimes \pred{C}\right) \circ \left( \Delta
_{C}\otimes \pred{C}\otimes \pred{C}\right) \circ \left( C\otimes {%
\pred{C}}\otimes m_\pred{C}\right) \\
&=&\mathrm{ev}_{C}\circ \left( C\otimes m_\pred{C}\right) \circ \left(
C\otimes \pred{C}\otimes m_\pred{C}\right)
\end{eqnarray*}%
so that $m_\pred{C}\circ \left( m_\pred{C}\otimes \pred{C}%
\right) =m_\pred{C}\circ \left( \pred{C}\otimes m_{\pred{C}%
}\right) .$ Let us check that $m_\pred{C}$ is unitary. We compute%
\begin{eqnarray*}
&&\mathrm{ev}_{C}\circ \left( C\otimes m_\pred{C}\right) \circ \left(
C\otimes u_\pred{C}\otimes \pred{C}\right) \circ \left( C\otimes l_{{%
\pred{C}}}^{-1}\right) \\
&=&\mathrm{ev}_{C}\circ \left( r_{C}\otimes \pred{C}\right) \circ \left(
C\otimes \mathrm{ev}_{C}\otimes \pred{C}\right) \circ \left( \Delta
_{C}\otimes \pred{C}\otimes \pred{C}\right) \circ \left( C\otimes u_{{%
\pred{C}}}\otimes \pred{C}\right) \circ \left( C\otimes l_{\pred{C}%
}^{-1}\right) \\
&=&\mathrm{ev}_{C}\circ \left( r_{C}\otimes \pred{C}\right) \circ \left(
C\otimes \mathrm{ev}_{C}\otimes \pred{C}\right) \circ \left( C\otimes
C\otimes u_\pred{C}\otimes \pred{C}\right) \circ \left( \Delta
_{C}\otimes \mathbf{1}\otimes \pred{C}\right) \circ \left( C\otimes l_{{%
\pred{C}}}^{-1}\right) \\
&=&\mathrm{ev}_{C}\circ \left( r_{C}\otimes \pred{C}\right) \circ \left(
C\otimes \varepsilon _{C}\otimes \pred{C}\right) \circ \left( C\otimes r_{%
{C}}\otimes \pred{C}\right) \circ \left( \Delta _{C}\otimes \mathbf{1}%
\otimes \pred{C}\right) \circ \left( C\otimes l_\pred{C}^{-1}\right)
\\
&=&\mathrm{ev}_{C}\circ \left( r_{C}\otimes \pred{C}\right) \circ \left(
C\otimes \varepsilon _{C}\otimes \pred{C}\right) \circ \left( C\otimes
C\otimes l_\pred{C}\right) \circ \left( \Delta _{C}\otimes \mathbf{1}%
\otimes \pred{C}\right) \circ \left( C\otimes l_\pred{C}^{-1}\right)
\\
&=&\mathrm{ev}_{C}\circ \left( r_{C}\otimes \pred{C}\right) \circ \left(
C\otimes \varepsilon _{C}\otimes \pred{C}\right) \circ \left( \Delta
_{C}\otimes \pred{C}\right) \circ \left( C\otimes l_\pred{C}\right)
\circ \left( C\otimes l_\pred{C}^{-1}\right) \\
&=&\mathrm{ev}_{C}=\mathrm{ev}_{C}\circ \left( C\otimes \mathrm{Id}_\pred{C}\right)
\end{eqnarray*}

so that $m_\pred{C}\circ \left( u_\pred{C}\otimes \pred{C}%
\right) \circ l_\pred{C}^{-1}=\mathrm{Id}_\pred{C}.$ Similarly%
\begin{eqnarray*}
&&\mathrm{ev}_{C}\circ \left( C\otimes m_\pred{C}\right) \circ \left(
C\otimes \pred{C}\otimes u_\pred{C}\right) \circ \left( C\otimes r_{{%
\pred{C}}}^{-1}\right) \\
&=&\mathrm{ev}_{C}\circ \left( r_{C}\otimes \pred{C}\right) \circ \left(
C\otimes \mathrm{ev}_{C}\otimes \pred{C}\right) \circ \left( \Delta
_{C}\otimes \pred{C}\otimes \pred{C}\right) \circ \left( C\otimes {%
\pred{C}}\otimes u_\pred{C}\right) \circ \left( C\otimes r_{\pred{C}%
}^{-1}\right) \\
&=&\mathrm{ev}_{C}\circ \left( C\otimes u_\pred{C}\right) \circ \left(
r_{C}\otimes \mathbf{1}\right) \circ \left( C\otimes \mathrm{ev}_{C}\otimes
\mathbf{1}\right) \circ \left( \Delta _{C}\otimes \pred{C}\otimes \mathbf{%
1}\right) \circ r_{C\otimes \pred{C}}^{-1} \\
&=&\varepsilon _{C}\circ r_{{C}}\circ \left( r_{C}\otimes \mathbf{1}\right)
\circ \left( C\otimes \mathrm{ev}_{C}\otimes \mathbf{1}\right) \circ \left(
\Delta _{C}\otimes \pred{C}\otimes \mathbf{1}\right) \circ r_{C\otimes {%
\pred{C}}}^{-1} \\
&=&\varepsilon _{C}\circ r_{{C}}\circ \left( C\otimes \mathrm{ev}_{C}\right)
\circ \left( \Delta _{C}\otimes \pred{C}\right) \circ r_{C\otimes
{\pred{C}}}\circ r_{C\otimes \pred{C}}^{-1} \\
&=&r_{\mathbf{1}}\circ \left( \varepsilon _{C}\otimes \mathbf{1}\right)
\circ \left( C\otimes \mathrm{ev}_{C}\right) \circ \left( \Delta _{C}\otimes
\pred{C}\right) \\
&=&l_{\mathbf{1}}\circ \left( \mathbf{1}\otimes \mathrm{ev}_{C}\right) \circ
\left( \varepsilon _{C}\otimes C\otimes \pred{C}\right) \circ \left(
\Delta _{C}\otimes \pred{C}\right) \\
&=&\mathrm{ev}_{C}\circ l_{C\otimes \pred{C}}\circ \left( \varepsilon
_{C}\otimes C\otimes \pred{C}\right) \circ \left( \Delta _{C}\otimes {%
\pred{C}}\right) \\
&=&\mathrm{ev}_{C}\circ \left( l_{C}\otimes \pred{C}\right) \circ \left(
\varepsilon _{C}\otimes C\otimes \pred{C}\right) \circ \left( \Delta
_{C}\otimes \pred{C}\right) =\mathrm{ev}_{C}=\mathrm{ev}_{C}\circ \left(
C\otimes \mathrm{Id}_\pred{C}\right)
\end{eqnarray*}%
so that $m_\pred{C}\circ \left( \pred{C}\otimes u_{\pred{C}%
}\right) \circ r_\pred{C}^{-1}=\mathrm{Id}_\pred{C}$.
\end{proof}

In order to attach a right $\pred{C}$-module to any right $C$-comodule, as in the classical setting, notice that for every $M\in \mathcal{M}$ 
we can consider the morphism%
\begin{equation}
\mathcal{M}\left( M,M\otimes C\right) \rightarrow \mathcal{M}\left( M\otimes
\pred{C},M\right),\qquad\rho \mapsto \mu _{\rho }  \label{form:defpurho1}
\end{equation}%
where $\mu _{\rho }$ denotes the composition
\begin{equation*}
M\otimes \pred{C} \xrightarrow{\rho \otimes \pred{C}} M\otimes
C\otimes \pred{C} \xrightarrow{M\otimes \mathrm{ev}_{C}} M\otimes
\mathbf{1} \xrightarrow{r_{M}} M.
\end{equation*}
Recall also that a functor $F \colon \mathcal{C} \to \mathcal{D}$ is \emph{conservative} if it is reflects isomorphisms, i.e., if $g \colon X \to Y$ is a morphism in $\mathcal{C}$ such that $F(g)$ is an isomorphism in $\mathcal{D}$, then $g$ is an isomorphism in $\mathcal{C}$.

\begin{proposition}
\label{pro:L}
Let $\left( \mathcal{M}%
,\otimes ,\mathbf{1}\right) $ be a monoidal category and let $%
\left( C,\Delta _{C},\varepsilon _{C}\right) $ be a coalgebra in $\mathcal{M}
$. Assume that $C\ $is left pre-rigid with left pre-dual $\pred{C}$.
Then we have a conservative and faithful functor $$L:\mathcal{M}^{C}\rightarrow {\mathcal{M}_\pred{C}}%
,\qquad\left( {N},\rho \right) \mapsto \left( {N},\mu _{\rho }\right),\qquad f\mapsto f .$$
Moreover, if all the following conditions hold:
\begin{enumerate}[label=\arabic*)]
\item the functor $\left( -\right) \otimes \pred{C}:\mathcal{M} \rightarrow \mathcal{M}$ preserves colimits (e.g.\ it is a left adjoint),
\item the functor $\left( -\right) \otimes C:\mathcal{M} \rightarrow \mathcal{M}$ preserves filtered colimits (e.g.\ it is a left adjoint) and
\item $\Mm$ is locally presentable,
\end{enumerate}
then $L:\Mm^{C}\rightarrow \Mm_\pred{C}$ has a right adjoint $R$ and $L$ is comonadic.
\end{proposition}

\begin{proof} 
$\pred{C}$ is an algebra in $\mathcal{M}$ by \cref{lem:*algebra}, hence $- \otimes \pred{C}$ is a monad on $\mathcal{M}$. The evaluation $\text{ev}_C \colon C \otimes \pred{C} \to \mathbf{1}$ induces a natural transformation $- \otimes \text{ev}_C \colon - \otimes C \otimes \pred{C} \to \id_{\mathcal{M}}$ which makes $\mathcal{P} \coloneqq (- \otimes \pred{C}, - \otimes C, - \otimes \text{ev}_C)$ a pairing between the monad $- \otimes \pred{C}$ and the comonad $- \otimes C$ in the sense of \cite[\S3.2]{MesaWis}: formulas \eqref{form:dualmult} and \eqref{form:dualun} entail the commutativity of \cite[(3.1)]{MesaWis}.
\begin{invisible}
In fact, the pentagon in \cite[(3.1)]{MesaWis} commutes if and only if
\[(Y \otimes \text{ev}_C) \circ (f \otimes \pred{C}) \circ (X \otimes m_{\pred{C}}) = (Y \otimes \text{ev}_C) \circ (Y \otimes C \otimes \text{ev}_C \otimes \pred{C}) \circ (Y \otimes \Delta \otimes \pred{C} \otimes \pred{C}) \circ (f \otimes \pred{C} \otimes \pred{C})\]
and the triangle commutes if and only if
\[(Y \otimes \text{ev}_C) \circ (f \otimes \pred{C}) \circ (X \otimes u_{\pred{C}}) = (T \otimes \varepsilon) \circ f.\]
\end{invisible}%
Then the functor $L$ of the statement is exactly the functor $\Phi^{\mathcal{P}}$ of \cite[\S3.6]{MesaWis} and it is faithful as it acts as the identity on morphisms.
Moreover, the last claim follows directly from \cite[\S3.7]{MesaWis}. 
\end{proof}

Recall that a monoidal category is called \emph{biclosed} if it is both left and right closed.

\begin{corollary}\label{cor:Lbiclosed} 
Let $\left( \mathcal{M},\otimes ,\mathbf{1}\right) $ be a biclosed monoidal category and let $%
\left( C,\Delta _{C},\varepsilon _{C}\right) $ be a coalgebra in $\mathcal{M}$. If $\Mm$ is locally presentable, then the functor $L:\Mm^{C}\rightarrow \Mm_\pred{C}$ has a right adjoint $R$.
\end{corollary}

\begin{proof} 
Since $\mathcal{M}$ is left closed, $C$ is left pre-rigid with left pre-dual $\pred{C}$. The property of being right closed implies that $(-)\otimes\pred{C}:\mathcal{M}\to\mathcal{M}$ and $-\otimes C \colon \mathcal{M} \to \mathcal{M}$ are left adjoints and so they preserve colimits and we can apply the second part of \cref{pro:L}.
\end{proof}

Having worked through the necessary groundwork, we're now ready to progress from Hopf algebras to bialgebras.

\subsection{The bialgebra case}\label{sec:bialg}

Let $B$ be a bialgebra with multiplication $m$, unit $u$, comultiplication $\Delta$ and counit $\varepsilon$. The category $\Mm={\mathfrak{M}_B}$ of right $B$-modules is monoidal with tensor product $\otimes =\otimes
_{\Bbbk }$ and unit object $\Bbbk $. Moreover $\mathfrak{M}_{B}$  is biclosed, i.e.~it is both  left and right closed (see e.g. \cite[Lemma 1.5]{Saracco-PreFrob}).
Explicitly, given an object $X$ in $\mathfrak{M}_{B}$, the right adjoint of $X\otimes(-):\mathfrak{M}_{B}\to\mathfrak{M}_{B}$ is given by
\[[X,-]\coloneqq \mathfrak{M}_{B}(X_\bullet\otimes B_\bullet, -)\]
where the two bullets indicate the diagonal action. For every $X$ and $Y$ in $\mathfrak{M}_{B}$,
the right $B$-module structure of $[X,Y]$ is given by $(fa)(x\otimes b)=f(x\otimes ab)$, for every $f\in [X,Y],x\in X,a,b\in B$.
The bijection that yields the adjunction is
\[\Phi_{Y,Z}:\mathfrak{M}_{B}(X_\bullet\otimes Y_\bullet,Z)\to \mathfrak{M}_{B}(Y,[X,Z])\]
where $\Phi_{Y,Z}(f)(y)(x\otimes b)=f(x\otimes yb)$ for every $f\in \mathfrak{M}_{B}(X\otimes Y,Z), x\in X, y\in Y,b\in B$ while
$\Phi_{Y,Z}^{-1}(g)(x\otimes y)=g(y)(x\otimes 1)$ for every $g\in \mathfrak{M}_{B}(Y,[X,Z]), x\in X, y\in Y$.

The unit of the adjunction is $\eta_Y\coloneqq \Phi_{Y,X\otimes Y}(\id_{X\otimes Y})$ i.e.
\[\eta_Y:Y\to [X,X\otimes Y],\qquad y\mapsto\left[x\otimes b\mapsto x\otimes yb\right].\]
The counit of the adjunction is $\epsilon_Z\coloneqq \Phi_{[X,Z],Z}^{-1}(\id_{[X,Z]})$ i.e.
\[\epsilon_Z:X\otimes [X,Z]\to Z\qquad x\otimes f\mapsto f(x\otimes 1).\]

In view of \cref{rmk:closed=>prig}, the left closedness of $\mathfrak{M}_{B}$ implies that $\mathfrak{M}_{B}$ is left pre-rigid with left pre-dual given by
\[\pred{X}\coloneqq [X,\Bbbk]=\mathfrak{M}_{B}(X_\bullet\otimes B_\bullet, \Bbbk).\]
The evaluation is $\ev_X\coloneqq \Phi_{{\pred{X}},\Bbbk}^{-1}(\id_{\pred{X}})=\epsilon_\Bbbk$ i.e.  $$\ev_X:X\otimes {}\pred{X}\to\Bbbk,\quad x\otimes f\mapsto f(x\otimes 1).$$

We now derive, in the present setting, the algebraic structure of the pre-dual obtained in  \cref{lem:*algebra} and the related functor $L$ of \cref{pro:L}.

\begin{proposition}
\label{prop:predBalg}
Let $B$ be a bialgebra. Then the pre-dual $\pred{B}=\mathfrak{M}_{B}(B_\bullet\otimes B_\bullet, \Bbbk)$ becomes a right $B$-module algebra where, given $f,g\in{}\pred{B},a,b\in B$, we have
\begin{equation}\label{eq:algstruct}
(f \conv g)(a\otimes b) = f(a_2\otimes b_1)g(a_1\otimes b_2),\qquad
1_{\pred{B}}(a\otimes b)=\varepsilon_B(a)\varepsilon_B(b).
\end{equation}
Moreover the functor
\begin{equation}\label{eq:functorL}
L \colon \mathfrak{M}_{B}^{B}\rightarrow (\mathfrak{M}_{B})_{\pred{B}},\qquad (N,\rho) \mapsto (N,\mu _{\rho }),
\end{equation}
has a right adjoint. Here  $n \triangleleft f\coloneqq  \mu_{\rho}(n\otimes f)= n_0f(n_1\otimes 1)$, for every $n\in N$ and $f\in {\pred{B}}$.
\end{proposition}

\begin{proof} 
First of all, since $B$ is clearly a coalgebra in $\mathfrak{M}_{B}$, we can apply \cref{lem:*algebra} to obtain a right $B$-module algebra structure on the pre-dual $\pred{B}$.
For $f,g\in{}\pred{B}$, we set $f \conv g \coloneqq m_{\pred{B}}(f\otimes g)$. Let us show this is explicitly given by the left-hand side of \eqref{eq:algstruct}. In fact, $m_{\pred{B}}$ is the unique right $B$-linear morphism defined by \eqref{form:dualmult}, by construction. The latter means that we have $\ev_B(a\otimes (f \conv g)) = \ev_B(a_1\otimes g)\ev_B(a_2\otimes f)$ for every $a\in B$ or, equivalently, $(f \conv g)(a\otimes 1)=g(a_1\otimes 1)f(a_2\otimes 1)$.
Now, right $B$-linearity of $m_{\pred{B}}$ entails that
\begin{align*}
    (f \conv g)(a\otimes b) & =((f \conv g)b)(a\otimes 1)=\big(\left(fb_1\right) \conv \left(gb_2\right)\big)(a\otimes 1) \\
    & = (fb_1)(a_2\otimes 1)(gb_2)(a_1\otimes 1) = f(a_2\otimes b_1)g(a_1\otimes b_2),
\end{align*}
which is what we were aiming for.
Similarly $1_{\pred{B}}$ is uniquely defined by \eqref{form:dualun} i.e.\ $\ev_B(b\otimes 1_{\pred{B}})=\varepsilon_B(b)$, that is $1_{\pred{B}}(b\otimes 1)=\varepsilon_B(b)$. Right $B$-linearity, in turn, entails that $1_{\pred{B}}(a\otimes b)=(1_{\pred{B}}b)(a\otimes 1)=1_{\pred{B}}(a\otimes 1)\varepsilon_B(b)=\varepsilon_B(a)\varepsilon_B(b)$.

Secondly, we can apply \cref{pro:L} to find the functor $L:\mathfrak{M}_{B}^{B}\rightarrow (\mathfrak{M}_{B})_{\pred{B}}$ which explicitly acts by $L: (N,\rho) \mapsto (N,\mu _{\rho }),$
where  $\mu_{\rho}$ is given by $n \triangleleft f = \mu_{\rho}(n\otimes f)=n_0\ev_B(n_1\otimes f)=n_0f(n_1\otimes 1)$ for every $n\in N$ and $f\in {\pred{B}}$.

Finally, we note that the forgetful functor $\mathfrak{M}_{B}\to \mathfrak{M}$ is monadic. The corresponding monad is $\top\coloneqq (-)\otimes B:\mathfrak{M}\to \mathfrak{M}$ which clearly preserves colimits. By \cite[Theorem 5.5.9]{Borceux2}, we get that $\Mm={\mathfrak{M}_{B}}\cong \mathfrak{M}_\top$ is locally presentable, as $\mathfrak{M}$ is. Then, by \cref{cor:Lbiclosed}, we conclude that the functor $L:\mathfrak{M}_{B}^{B}\rightarrow (\mathfrak{M}_{B})_{\pred{B}}$ has a right adjoint.
\end{proof}

As we will see, a concrete description of the integrals can be achieved without requiring an explicit formula for the right adjoint $R$. However, the reader interested in having one can find it in \cref{ssec:RightAdj}.
 For the convenience of the interested reader and for future reference, let us extrapolate the following result, which went unnoticed in the foregoing argument.

\begin{proposition}\label{prop:HopfModLocPres}
For a bialgebra $B$, the category $\mathfrak{M}_{B}^{B}$ of $B$-Hopf modules is locally presentable.
\end{proposition}

\begin{proof} 
A step in the proof of \cref{pro:L} consists in adapting \cite[\S2.78, Theorem]{AR} to show that the Eilenberg-Moore category of coalgebras for a filtered-colimit preserving comonad on a locally presentable category is locally presentable. In the present setting, the comonad $-\otimes B \colon \mathfrak{M}_B \to \mathfrak{M}_B$ is a left adjoint (because $\mathfrak{M}_B$ is biclosed) and hence cocontinuous. Its Eilenberg-Moore category is exactly $\mathfrak{M}_{B}^{B}$. Since $\mathfrak{M}_B$ is locally presentable, $\mathfrak{M}_{B}^{B}$ is locally presentable too.
\end{proof}

As mentioned, we are mainly interested in defining and studying left integrals on an arbitrary bialgebra $B$, not necessarily endowed with an antipode. The idea is to replace the diagram \eqref{diag:func} by the following one
\begin{equation}\label{diag:funcbialg}
\xymatrix{{\mathfrak{M}}\ar@<.5ex>[r]^-F& \mathfrak{M}_{B}^{B}\ar@<.5ex>[l]^--{G}\ar@<.5ex>[r]^-L& (\mathfrak{M}_{B})_{\pred{B}}\ar@<.5ex>[l]^-{R} }
\end{equation}
where $F:{\mathfrak{M}}\to \mathfrak{M}_B^B,{ M}\mapsto {M\otimes {B^\bullet_\bullet}}$, is the fully faithful functor involved in the structure theorem for $B$-Hopf modules and whose right adjoint is $G=(-)^{\mathrm{co} B}$. Then, to define the space of left integrals on $B$ as $GR(\pred{B})$, like in the classical case.


In order to be able to easily compute $GR$, whence $GR(\pred{B})$, let us recall that $LF\dashv GR$. Thus, by uniqueness of the right adjoint, it suffices to exhibit a right adjoint for $LF$ to get a complete description for $GR$. To this aim, note that $LF=(-)\otimes B:\mathfrak{M}\to (\mathfrak{M}_{B})_{\pred{B}}$ where  $B$ is regarded as a right $B$-module via the regular action and as a right $\pred{B}$-module via
$b\triangleleft f\coloneqq b_1f(b_2\otimes 1)$, for every $b\in B,f\in{}\pred{B}$. We also need the following remark.

\begin{remark}\label{rmk:semi-direct}
Let $B$ be  bialgebra and let  $A$ be a right $B$-module algebra. Then there is  a category isomorphism $(\mathfrak{M}_{B})_{A}\cong \mathfrak{M}_{B\ltimes A}$,  where $B\ltimes A$ is the semi-direct product algebra whose product is given, for every $b,b'\in B,a,a'\in A$, by \[(b\ltimes a)(b'\ltimes a')=bb'_1\ltimes(ab'_2)a',\] see \cite[Paragraph 1.7]{Som} for the Hopf case  on the left-hand side.
\begin{invisible}
Se vogliamo farlo per le quasi-bialgebre, basta guardare \cite[Proposition 2.16]{BPV}. Vediamo però la dim per le Hopf.
 If $M$ an object in  $(\mathfrak{M}_{B})_{A}$ we can make of it an object in  $\mathfrak{M}_{B\ltimes A}$ by setting $m(b\ltimes a)\coloneqq (mb)a$. Indeed, we have $(m(b\ltimes a))(b'\ltimes a')=(((mb)a)b')a'=((mbb'_1)(ab'_2))a'=(mbb'_1)((ab'_2)a')=m(bb'_1\ltimes(ab'_2)a')=m((b\ltimes a)(b'\ltimes a'))$. Conversely, given $M$ in $\mathfrak{M}_{B\ltimes A}$ we regard it as an object in $(\mathfrak{M}_{B})_{A}$ by setting $mb\coloneqq m(b\ltimes 1_A)$ and $ma\coloneqq m(1_B\ltimes a)$. Indeed it is clear that $M\in \mathfrak{M}_{A}$. Moreover $(mb)b'=m(b\ltimes 1_A)(b'\ltimes 1_A)=m(bb'_1\ltimes1_Ab'_2)=m(bb'_1\ltimes1_A\varepsilon_B(b'_2))=m(bb'\ltimes1_A)=m(bb')$ and
 $(ma)b=mm(1_B\ltimes a)(b\ltimes 1_A)=m(b_1\ltimes ab_2)=(mb_1)(ab_2)$ which means that the right $A$-action is right $B$-linear.
\end{invisible}\end{remark}

\begin{proposition}
Let $B$ be a bialgebra and let $A$ be right $B$-module algebra. Assume that the right regular $B$-module $B$ becomes an object in $(\mathfrak{M}_{B})_{A}$. Then the functor $(-)\otimes B:\mathfrak{M}\to (\mathfrak{M}_{B})_{A}$ has a right adjoint $(\mathfrak{M}_{B})_{A}(B,-):(\mathfrak{M}_{B})_{A}\to \mathfrak{M}$.
\end{proposition}

\begin{proof}
The claimed adjunction is the usual tensor-hom adjunction
up to the identification given by the category isomorphism $(\mathfrak{M}_{B})_{A}\cong \mathfrak{M}_{B\ltimes A}$ recalled in \cref{rmk:semi-direct}.
\end{proof}

We now provide a more handy description of the functor $(\mathfrak{M}_{B})_{A}(B,-):(\mathfrak{M}_{B})_A\to \mathfrak{M}.$

\begin{lemma}
Let $B$ be a bialgebra and let $A$ be right $B$-module algebra. Assume that the right regular $B$-module $B$ becomes an object in $(\mathfrak{M}_{B})_{A}$ with respect to an $A$-module structure $\bra \colon B \otimes A \to B$.
For every object $X$ in $(\mathfrak{M}_{B})_{A}$, the assignment $\xi\mapsto \xi(1)$ yields a $\Bbbk$-linear isomorphism \[(\mathfrak{M}_{B})_{A}(B,X)\to X^{A}\coloneqq \{x\in X\mid (xb)a = x(b \bra a),\text{ for all }a\in A,b\in B\}.\]
\end{lemma}

\begin{proof}
Let $\xi \in (\mathfrak{M}_{B})_{A}(B,X)$. Since $\xi$ is right $B$-linear it is uniquely determined by $\xi(1)\in X$. Moreover $\xi$ is right $A$-linear, i.e.~$\xi (b \bra a) = \xi (b)a$ for every $b\in B,a\in A$. Thus $\xi (1)(b \bra a) = (\xi (1)b)a$ that is $\xi(1)$ belongs to the space $X^{A}$ of ``right $A$-invariant'' elements in $X$. Thus we have the claimed bijection.
\end{proof}

Summing up, for every object $X$ in $(\mathfrak{M}_{B})_{\pred{B}}$, we have $GR(X)\cong (\mathfrak{M}_{B})_{\pred{B}}(B,X)\cong X^{\pred{B}}$.

By taking $X={}\pred{B}$ and comparing with \cite[Definition 5.1.1]{DNC}, we are now going to give a better presentation of the vector space
\[\textstyle GR(\pred{B})\cong ({\pred{B}})^{\pred{B}}=\{f\in {}\pred{B}\mid (fb) \conv g=f(b\triangleleft g),\text{ for all }g\in {}\pred{B},b\in B\}.\]

\begin{lemma}\label{lem:fotrev} Let $f\in {}\pred{B}$. Then $f\in ({\pred{B}})^{\pred{B}}$ if and only if, for every $g\in {}\pred{B},x,y\in B$, one of the following equivalent conditions holds true.
\begin{equation}\label{form:integ1}
f(x_2\otimes yz_1)g(x_1\otimes z_2)=f(x\otimes y_1z)g(y_2\otimes 1),\quad \text{ for every } z\in B.
\end{equation}
\begin{equation}\label{form:integ2}
f(x_2\otimes y)g(x_1\otimes 1)=f(x\otimes y_1)g(y_2\otimes 1).
\end{equation}
\end{lemma}

\begin{proof}
By definition, $f\in ({\pred{B}})^{\pred{B}}$ if and only if, for every $g\in {}\pred{B},y\in B$, one has $(fy) \conv g=f(y\triangleleft g)$ and when we evaluate this equality on $x\otimes z$, with $x,z\in B$, we get \eqref{form:integ1}.

Let us check that \eqref{form:integ1} and  \eqref{form:integ2} are equivalent. If we take $z=1$ in \eqref{form:integ1}, we get \eqref{form:integ2}. Conversely, if \eqref{form:integ2} holds true so does \eqref{form:integ1}, as
\begin{align*}
f(x_2\otimes by_1)g(x_1\otimes y_2)
&=f(x_2\otimes by_1)(gy_2)(x_1\otimes 1)
\overset{\eqref{form:integ2}}
=f(x\otimes b_1y_1)(gy_3)(b_2y_2\otimes 1)\\
&=f(x\otimes b_1y_1)g(b_2y_2\otimes y_3)
=f(x\otimes b_1y)g(b_2\otimes 1).\qedhere
\end{align*}
\end{proof}

\begin{remark}
We noticed that $B$ is a right $\pred{B}$-module via
$b\triangleleft g\coloneqq b_1g(b_2\otimes 1)$, for every $b\in B,g\in \pred{B}$.  It is also a left $\pred{B}$-module via
$g\triangleright b\coloneqq g(b_1\otimes 1)b_2$ and even a $\pred{B}$-bimodule. With respect to these two actions, condition \eqref{form:integ2} rewrites as
$f(g\triangleright x\otimes y)=f(x\otimes y\triangleleft g)$.
\end{remark}

In order to give an even better presentation of $({\pred{B}})^{\pred{B}}$, we take advantage of the following notation introduced in \cite[Notation 2.1]{ArMeSa} in case $Y=B$.

\begin{notation}
\label{not:oslash}
Let $B$ be a bialgebra. For every $X,Y$ in $\mathfrak{M}_B$, we denote by $X\oslash Y$ the quotient $\frac{X_\bullet\otimes Y_\bullet}{(X_\bullet\otimes Y_\bullet) \cdot B^+}$, where $B^+=\ker(\varepsilon_B)$ is the augmentation ideal.  We also denote by $x\oslash y$ the equivalence class of $x\otimes y$, for all $x\in X,y\in Y$, so that $xb_1\oslash yb_2=x\oslash y\varepsilon(b)$, for every $b\in B$. Symmetrically, for every $U,V$ in $\prescript{}{B}{\mathfrak{M}}$, we denote by $U\obslash V$ the quotient $\frac{\prescript{}{\bullet}{U} \otimes \prescript{}{\bullet}{V}}{B^+ \cdot (\prescript{}{\bullet}{U} \otimes \prescript{}{\bullet}{V})}$.
\end{notation}

\begin{remark}
\label{rmk:BotBact}
Recall from \cite[Proposition 2.2]{ArMeSa} that, if $B$ is a bialgebra, then $B\oslash B$ is a coalgebra in $_{B \otimes B^\cop}\mathfrak{M}$ with comultiplication, counit and left $B \otimes B$-module structure given by
\begin{equation}\label{def:oslashcoalg}
   \Delta(x\oslash y) = (x_1\oslash y_2) \otimes (x_2\oslash y_1),
   \quad
   \varepsilon(x\oslash y)=\varepsilon(x)\varepsilon(y)
   \quad \text{and}\quad
   (a \otimes b) \cdot (x\oslash y) = ax\oslash by,
   \end{equation}
   for all $x,y,a,b\in B$. In particular the canonical projection $B\otimes B^\cop\to B\oslash B,\;x\otimes y\mapsto x\oslash y,$ is a coalgebra map, so that its kernel $(B\otimes B)\cdot B^+ = (B\otimes B)\Delta(B^+)$ is a coideal in $B\otimes B^\cop$. In what follows, we will often omit the $\cdot$ symbol and simply write $(B\otimes B)B^+$ or $(B\otimes B)\Delta(B^+)$ indifferently.
\end{remark}

We can now  provide an alternative description of the pre-dual by means of Notation \ref{not:oslash}. 


\begin{lemma}\label{lem:oslash}
Let $B$ be a bialgebra. For every $X$ in $\mathfrak{M}_B$, we have a linear isomorphism
\[\pred{X}=\mathfrak{M}_{B}(X_\bullet\otimes B_\bullet, \Bbbk)\overset{\sim}{\longrightarrow} \mathfrak{M}(X\oslash B,\Bbbk) = (X\oslash B)^*,\quad f\mapsto [x\oslash a\mapsto f(x\otimes a)].\]
Through this isomorphism, the right $B$-module algebra $\pred{B}$ corresponds to the dual algebra ${^*(B \oslash B)} = {(B \oslash B)^*}^\op$ of $B\oslash B$, considered as a coalgebra and as a left $B^\cop$-module. 
\end{lemma}

\begin{proof}
Consider the counit $\varepsilon:B\to \Bbbk$. The restriction of scalars functor $\varepsilon_*:\mathfrak{M}\to \mathfrak{M}_B$ has a left adjoint given by the extension of scalars (or induction) functor $\varepsilon^*=(-)\otimes_B \Bbbk:\mathfrak{M}_B\to \mathfrak{M}$.
  Thus
  $\pred{X}=\mathfrak{M}_{B}(X_\bullet\otimes B_\bullet, \Bbbk)
  =\mathfrak{M}_{B}(X_\bullet\otimes B_\bullet, \varepsilon_*(\Bbbk))
  \cong\mathfrak{M}(\varepsilon^*(X_\bullet\otimes B_\bullet), \Bbbk)$,
  so that the pre-dual $\pred{X}$ is isomorphic to the linear dual of \[
  \varepsilon^*(X_\bullet\otimes B_\bullet)=(X_\bullet\otimes B_\bullet)\otimes_B \Bbbk\cong\frac{X_\bullet\otimes B_\bullet}{(X_\bullet\otimes B_\bullet)B^+} = X\oslash B.\]
  As a consequence, we get the bijection in the statement and so $\pred{B}$ is isomorphic, as a vector space, to the linear dual $\mathfrak{M}(B\oslash B,\Bbbk) = (B \oslash B)^*$. By comparing the right $B$-module algebra structure on $\pred{B}$ from \cref{prop:predBalg} with the left $B^\cop$-module coalgebra structure on $B \oslash B$ from \cref{rmk:BotBact}, it becomes clear that the former is the dual algebra of the latter as claimed. 
\end{proof}


\begin{lemma}\label{rem:oslash}
Via the isomorphism given in \cref{lem:oslash}, the elements in $({\pred{B}})^{\pred{B}}$ correspond to the elements $f\in\mathfrak{M}(B\oslash B,\Bbbk)$ that fulfil anyone of the following equivalent conditions.
 \begin{gather}
f(x_2\oslash yz_1)g(x_1\oslash z_2)=f(x\oslash y_1z)g(y_2\oslash 1),\text{ for all } x,y,z\in B, g\in \mathfrak{M}(B\oslash B,\Bbbk). \label{form:integosl1} \\
f(x_2\oslash yz_1)x_1\oslash z_2=f(x\oslash y_1z)y_2\oslash 1,\text{ for all } x,y,z\in B. \label{form:integosl2} \\
x_1 f(x_2\oslash y)\oslash 1=f(x\oslash y_1)y_2\oslash 1,\text{ for all } x,y\in B. \label{form:integosl3}
\end{gather}
\end{lemma}

\begin{proof}
If we look at an element $f$ in $({\pred{B}})^{\pred{B}}$ as an element in $\mathfrak{M}(B\oslash B,\Bbbk)$ via the mentioned isomorphism, then \eqref{form:integ1} rewrites  as \eqref{form:integosl1}.
   Since the latter holds for every $\Bbbk$-linear map $g$, it is equivalent to \eqref{form:integosl2}.
Again, by taking $z=1$ in \eqref{form:integosl2}, one gets this is equivalent to \eqref{form:integosl3} in a similar manner as done in the proof of \cref{lem:fotrev}.
\end{proof}

In view of \cref{lem:oslash}, henceforth we adopt the identification $\pred{B} = {(B\oslash B)^*}$ as vector spaces and $\pred{B} = {^*(B \oslash B)}$ as algebras. In particular, for all $f,g \in (B \oslash B)^*$, $x,y \in B$ we write
\[(f \conv g)(x \oslash y) = f(x_2 \oslash y_1)g(x_1 \oslash y_2).\]

The paragraph following \cref{prop:HopfModLocPres}, read in light of \cref{rem:oslash}, lead to the next definition, which introduces the core notion of the present paper.

\begin{definition}\label{def:oslash} A \emph{left integral} $f$ in $\pred{B}$ is an element in $(B\oslash B)^*$ such that \eqref{form:integosl3} holds true.
We denote by  $\int_l \pred{B}$ the \emph{space of left integrals in} $\pred{B}$ so that
\[\textstyle \int_l \pred{B}=\{f:B\oslash B\to\Bbbk\mid x_1 f(x_2\oslash y)\oslash 1=f(x\oslash y_1)y_2\oslash 1,\text{ for all } x,y\in B\}.\]
A left integral $\lambda\in \int_l \pred{B}$ will be called a \emph{total integral} in case $\lambda(1\oslash 1)=1$.
\end{definition}

\begin{remark}\label{rem:LRints}
Note that if we choose $y=1$ in \eqref{form:integosl2}, then we get $f(x_2\oslash z_1)x_1\oslash z_2=f(x\oslash z)1\oslash 1$ for all $x,z\in B.$ This means that $f$ is, in particular, a usual left integral in the augmented algebra $(B\oslash B)^*$, with augmentation $\varepsilon_{(B\oslash B)^*} \colon (B\oslash B)^*\to \Bbbk,f\mapsto f(1\oslash 1).$ That is to say, $g*f = g(1 \oslash 1)f$ for all $g \in (B\oslash B)^*$.
Summing up, $\int_{l} {^{\star }B}
\subseteq \int_{l}\left( B\oslash B\right) ^{\ast }$ for any bialgebra $B$.
\end{remark}

\cref{sec:partcases} and \cref{sec:examples} are devoted to present results on, and examples of, integrals on bialgebras.

\subsection{The map \texorpdfstring{$i_B$}{i\_B}}\label{ssec:iB}
We already know from \cite[\S2.1]{ArMeSa} that for every bialgebra $B$, one can consider the canonical map $i_B \colon B\to B\oslash B,x\mapsto x\oslash 1$.
We summarise here some properties of $i_B$, as it is helpful in the study of integrals in $\pred{B}$, too.\medskip

First of all, recall that a \emph{right Hopf algebra} is a bialgebra with a \emph{right antipode}  $S^{r}$, i.e.\ a right convolution inverse of $\mathrm{Id}_{B}$. Similarly, a \emph{left Hopf algebra} is a bialgebra with a \emph{left antipode} $S^{l}$. It was shown, e.g.\ in \cite[Section 3]{Sar21}, that the functors $F = \left( -\right) \otimes B_{\bullet }^{\bullet } \colon \mathfrak{M}\to\mathfrak{M}_{B}^{B}$ and $G = \left(
-\right) ^{\mathrm{co}B} \colon \mathfrak{M}_{B}^{B}\rightarrow \mathfrak{M}$ of diagram \eqref{diag:funcbialg}
are part of an adjoint triple $E\dashv F\dashv G$, where $E=\overline{\left( -\right) }^{B} \colon \mathfrak{M}_{B}^{B}\rightarrow \mathfrak{M}$ maps a $B$-Hopf module $M$ to  $\overline{M}^{B}=\frac{M}{MB^{+}}$ where $B^{+}=\ker (\varepsilon _{B})$.
The following result collects the essential points of \cite[Theorem 3.7]{Sar21} and \cite[Proposition 2.4]{ArMeSa} and it characterises the bijectivity of $i_B$.

\begin{proposition}
\label{prop:Frobenius}Let $B$ be a bialgebra. The following are
equivalent.

\begin{enumerate}[label=(\arabic*),ref={\itshape (\arabic*)},itemsep=0.1cm]
\item The coinvariant functor $\left( -\right) ^{\mathrm{co}B}:\mathfrak{M}%
_{B}^{B}\rightarrow \mathfrak{M}$ is Frobenius.


\item The map $i_{B}:B\rightarrow B\oslash B,b\mapsto b\oslash 1_{B},$ is invertible.

\item\label{item:1sided3} $B$ is right Hopf algebra with anti-multiplicative and
anti-comultiplicative right antipode $S^{r}$.
\end{enumerate}
Moreover, if \ref{item:1sided3} holds true, then $i_{B}^{-1}\left( x\oslash y\right)
=xS^{r}\left( y\right) $ for every $x,y\in B.$
\end{proposition}

In particular, $i_B$ is bijective if $B$ is a Hopf algebra.
However, we are also interested in more general cases in which $i_B$ is either injective or surjective. In this direction, the following result reports the key points of \cite[Proposition 2.10]{ArMeSa} which we are going to use in the sequel, for the convenience of the reader.

\begin{proposition}
\label{lem:iB-inj}
Let $B$ be a bialgebra and consider the map $%
i_{B}:B\rightarrow B\oslash B,b\mapsto b\oslash 1_{B}.$
If $B$ embeds into a bialgebra $C$ with $i_{C}$ injective, then so is $i_{B}$. In particular, if $B$ embeds into a Hopf algebra, then $i_B$ is injective.
%
%
%
%
%
\end{proposition}





In case $i_B$ is injective, we have an easier description of integrals in $\pred{B}$.

\begin{proposition}\label{pro:cosepiB}
Let $B$ be a bialgebra such that $i_B$ is injective. Then  $f\in \pred{B}$ is a left integral if and only if for all $x,y\in B$ we have
\begin{equation}\label{eq:futuresep}
x_1 f(x_2\oslash y)=f(x\oslash y_1)y_2.
\end{equation}
\end{proposition}

\begin{proof} Clearly a map $f\in \pred{B}$ that fulfils $x_1 f(x_2\oslash y)=f(x\oslash y_1)y_2$, also verifies $x_1 f(x_2\oslash y)\oslash 1=f(x\oslash y_1)y_2\oslash 1$ i.e. \eqref{form:integosl3} holds true, for all $x,y\in B$, which means it is a left integral. Conversely, if  $f$ is a left integral in $\pred{B}$, then \eqref{form:integosl3} holds true or, equivalently, $i_B(x_1 f(x_2\oslash y))=i_B(f(x\oslash y_1)y_2)$ and hence $x_1 f(x_2\oslash y)=f(x\oslash y_1)y_2$, for all $x,y\in B$, by injectivity of $i_B$.
\end{proof}

In particular, the easier description provided in \cref{pro:cosepiB} allows us to related the new integrals with the classical ones.

\begin{remark}
\label{rem:oldandnew}
    Let $B$ be a bialgebra. Since $B^*$ is an augmented algebra, one can consider usual left integrals in $B^*$, i.e.
    \[\textstyle \int_l B^* = \left\{\lambda \colon B \to \Bbbk \mid b_1\lambda(b_2) = 1 \lambda(b) \text{ for all } b \in B\right\}.
    \]
    If the canonical map $i_B \colon B \to B \oslash B$ is injective, one can also remark that any integral $f \in \int_l\pred{B}$ in the new sense induces an integral in the classical sense by precomposing it with $i_B$: if we set $\lambda \coloneqq f \circ i_B$, then
    \[b_1\lambda(b_2) = b_1f(b_2 \oslash 1) \stackrel{\eqref{eq:futuresep}}{=} 1f(b \oslash 1) = 1 \lambda(b)\]
    for all $b \in B$. In this case, we have a linear map
    \begin{equation}\label{eq:intcomparison}
        \textstyle \omega \colon \int_l\pred{B} \to \int_lB^*, \qquad f \mapsto f \circ i_B.
    \end{equation}
\end{remark}

As we will see in \cref{ssec:cocomm}, the surjectivity of $i_B$ has enlightening implications in the study of integrals on cocommutative bialgebras and, in particular, on monoid bialgebras (to be treated separately in a forthcoming publication). The following result reports the more precise characterization of the surjectivity of $i_B$ that appears in \cite[Proposition 2.13]{ArMeSa}, for the sake of future reference.

\begin{proposition}\label{lem:iBsu}
  The following assertions are equivalent for a bialgebra $B$.
  \begin{enumerate}
      \item The map $i_B:B\to B\oslash B,\, x\mapsto x\oslash 1$, is surjective.
    \item There is $S\in \mathrm{End}_\Bbbk(B)$  such that $1\oslash y= S(y)\oslash 1$, for every $y\in B$.
    \item There is $S\in \mathrm{End}_\Bbbk(B)$  such that $x\oslash y= xS(y)\oslash 1$, for every $x,y\in B$.
     \item For every $y\in B$ there is $y'\in B$ such that $1\oslash y= y'\oslash 1$.
     \item For every $x,y\in B$ there is $y'\in B$ such that $x\oslash y= xy'\oslash 1$.
    \end{enumerate}
\end{proposition}
Finally, recall from \cite[Remark 2.7]{ArMeSa} that if $B$ is a bialgebra and $f:B\to H$ is a bialgebra map into a Hopf algebra $H$ with antipode $S$, then we have a coalgebra map
\begin{equation}\label{eq:fhat}
\widehat f \colon B\oslash B \to H, \qquad x\oslash y \mapsto f(x)Sf(y),
\end{equation}
such that $\widehat{f} \circ i_B = f$.

\section{Integrals in cases of interest} \label{sec:partcases}

Here we investigate the existence of integrals in $\pred{B}$ in some contexts of interest and under additional assumptions on $B$.

\subsection{Integrals and coseparability} \label{ssec:cosep}
It is known that a Hopf algebra $H$ is coseparable as a coalgebra in $\mathfrak{M}_H$ if and only if it is cosemisimple, see e.g.\ \cite[Proposition 2.11]{AMS}. 
Moreover, by dual Maschke's Theorem, this is also equivalent to the existence of a usual left total integral $t\in H^*$, see e.g.\ \cite[Theorem 2.4.6]{Montgomery}.
Our first result shows how the coseparability of a bialgebra $B$ as a coalgebra in $\mathfrak{M}_B$ is close to the existence of a total integral in $\pred{B}$, but first we need to recall the following.

\begin{remark}
\label{rem:cosep}
    Let $B$ be a bialgebra, $D$ be a right $B$-module coalgebra and $C$ be a left $B$-module coalgebra. Recall that $D$ is coseparable in $\mathfrak{M}_B$ if and only if there is a $D$-bicomodule morphism $\theta \colon D\otimes D\rightarrow D$ in $\mathfrak{M}_B$ such that $\theta \circ \Delta _{D}=\mathrm{Id}$, where $D\otimes D$ is a $D$-bicomodule via $\Delta_D \otimes D$ and $D\otimes \Delta_D $. Symmetrically, $C$ is coseparable in $\prescript{}{B}{\mathfrak{M}}$ if and only if there is a $C$-bicomodule morphism $\vartheta \colon C\otimes C\rightarrow C$ in $\prescript{}{B}{\mathfrak{M}}$ such that $\vartheta \circ \Delta _{C}=\mathrm{Id}$, where $C\otimes C$ is a $C$-bicomodule via $\Delta_C \otimes C$ and $C\otimes \Delta_C $.

    It is well-known that to give a $D$-bicomodule morphism $\theta \colon D\otimes D\rightarrow D$ in $\mathfrak{M}$ is equivalent to give a morphism $f \colon D\otimes D\rightarrow \Bbbk $ in $ \mathfrak{M}$ such that
    \[
    x_1f(x_2\otimes y) = f(x\otimes y_1) y_2
    \]
    holds true for all $x,y\in D$ (and symmetrically for $C$). The correspondence is given by the assignments $\theta \mapsto f _{\theta } \coloneqq \varepsilon _{D}\theta $ and $f \mapsto \theta _{f}$ defined by $\theta _{f}(x\otimes y) \coloneqq x_{1}f (x_{2}\otimes y) = f (x\otimes y_{1})y_{2}$ for every $x,y\in D$.

    Now, if $\theta$ is right $B$-linear, then
    \[
    f_\theta\left(\left(x\otimes y\right)b\right) = \varepsilon_D\theta((x\otimes y)b)
     = \varepsilon_D(\theta(x\otimes y)b) = \varepsilon_D\theta(x\otimes y)\varepsilon_B(b) = f_\theta (x\otimes y)\varepsilon_B(b)
     \]
     i.e.\ $f_\theta\in \mathfrak{M}_B(D_{\bullet} \otimes D_{\bullet},\Bbbk) = \mathfrak{M}(D \oslash D,\Bbbk)$.
    Conversely, if $f \in \mathfrak{M}_B(D_{\bullet} \otimes D_{\bullet},\Bbbk) = \mathfrak{M}(D \oslash D,\Bbbk)$, then
    \[\theta_f((x\otimes y)b) = \theta_f(xb_1\otimes yb_2) = x_1b_1f(x_2b_2\oslash yb_3) = x_1b_1 f(x_2\oslash y)\varepsilon_B(b_2) = \theta_f(x\otimes y)b\]
    so that $\theta_f$ is a morphism in $\mathfrak{M}_B$. Hence, the existence of a $D$-bicomodule and right $B$-linear morphism $\theta \colon D\otimes D\rightarrow D$ is equivalent to the existence of a linear morphism $f \colon D\oslash D\rightarrow \Bbbk $ such that
    \[
    x_1f(x_2\oslash y) = f(x\oslash y_1) y_2 \qquad \text{ for all } x,y\in D.
    \]
    Symmetrically, for $C$, the existence of a $C$-bicomodule and left $B$-linear morphism $\vartheta \colon C\otimes C\rightarrow C$ is equivalent to the existence of a linear morphism $g \colon C\obslash C\rightarrow \Bbbk $ such that
    \[
    x_1g(x_2\obslash y) = g(x\obslash y_1) y_2 \qquad \text{ for all } x,y\in C.
    \]

\end{remark}

\begin{proposition}\label{prop:cosep}
Let $B$ be a bialgebra. Then $B$ is coseparable as a coalgebra in $\mathfrak{M}_B$ if and only if there is $f\in {\pred{B}}$ such that $f(1\oslash 1)=1$ and
\begin{equation}\label{fcosep}
   x_1f(x_2\oslash y)=f(x\oslash y_1) y_2 \qquad \text{ for all } x,y\in B.
  \end{equation}
  In particular,  if $B$ is coseparable as a coalgebra in $\mathfrak{M}_B$, then there is a total integral in $\pred{B}$.
\end{proposition}

\begin{proof}
By \cref{rem:cosep}, we know that $B$ is coseparable as a coalgebra in $\mathfrak{M}_B$ if and only if there is a $B$-bicomodule morphism $\theta \colon B\otimes B\to B$ in $\mathfrak{M}_B$ such that $\theta \circ \Delta_B = \id$ and we know that the existence of $\theta$ $B$-bicolinear and right $B$-linear is equivalent to the existence of $f \colon B \oslash B \to \Bbbk$ such that \eqref{fcosep} holds true.
 If we further require that $\theta \circ \Delta_B = \id$, then $f_\theta(1_B\oslash 1_B)=\varepsilon\theta\Delta_B(1_B)=\varepsilon_B(1_B)=1$. Conversely, if we further require that $f(1_B\oslash 1_B)=1$, then  $\theta_f\Delta_B(x)=\theta_f(x_1\otimes x_2)= x_1f(x_2\oslash x_3)= x_1f(1_B\oslash 1_B)\varepsilon(x_2)=x$ so that $\theta_f \circ \Delta_B = \id$.

 We have so proved that $B$ is coseparable as a coalgebra in $\mathfrak{M}_B$ if and only if there is $f\in \mathfrak{M}(B \oslash B,\Bbbk)$ such that $f(1\oslash 1)=1$ and \eqref{fcosep} holds true for every $x,y\in B$.

Concerning the second part of the statement, it is clear that if $f\in {\pred{B}}$ verifies $f(1\oslash 1)=1$ and \eqref{fcosep}, then \eqref{form:integosl3} is also true and hence $f$ is a total integral.
\end{proof}

\begin{theorem}[Dual Maschke-type theorem]\label{cor:totiBinj}
Let $B$ be a bialgebra such that $i_B$ is injective. Then there is a total integral in $\pred{B}$ if and only if $B$ is coseparable as a coalgebra in $\mathfrak{M}_B$.
\end{theorem}

\begin{proof}
    It follows from \cref{pro:cosepiB} and \cref{prop:cosep}.
\end{proof}


Since $B\oslash B$ is a coalgebra in $_{B\otimes B^\cop}\mathfrak{M}$  (see \cref{rmk:BotBact}), it is in particular a coalgebra both in $_{B}\mathfrak{M}$ and in $_{B^\cop}\mathfrak{M}$. Thus, we can wonder whether its cosepability therein could also be related to the existence of a total integral. \cref{cor:BoslBcosepint}, which follows from next result, provides a first answer to this question.


\begin{proposition}
\label{prop:BoslashBcosep}
Let $B$ be a bialgebra. Then
\begin{enumerate}[label=\roman*),ref={\itshape \roman*)},leftmargin=*]

\item\label{item:cosep1} $B\oslash B$ is coseparable as a coalgebra in $_{B}\mathfrak{M}$ if and only if there is $\beta \colon \left( B\oslash B\right) \obslash \left( B\oslash B\right) \rightarrow \Bbbk $ in $\mathfrak{M}$
such that $\beta \left( \left( 1\oslash 1\right) \obslash \left( 1\oslash
1\right) \right) =1$ and
\begin{equation}
u_{1}\beta (u_{2}\obslash v)=\beta (u\obslash v_{1})v_{2} \qquad \text{ for all } u,v\in B\oslash B, \label{oslcosep}
\end{equation}
where $\obslash$ is taken with respect to the left $B$-action.

\item\label{item:cosep2} $B\oslash B$ is coseparable as a coalgebra in $_{B^\cop}\mathfrak{M}$ if and only if there is $\gamma \colon \left( B\oslash B\right) \obslash \left( B\oslash B\right) \rightarrow \Bbbk $ in $\mathfrak{M}$ such that $\gamma \left( \left( 1\oslash 1\right) \obslash \left( 1\oslash 1\right) \right) =1$ and \eqref{oslcosep}, i.e.
\begin{equation}\label{eq:oslcosepcop}
u_{1}\gamma (u_{2}\obslash v)=\gamma (u\obslash v_{1})v_{2} \qquad \text{ for all } u,v\in B\oslash B,
\end{equation}
where $\obslash$ is taken with respect to the left $B^\cop$-action.
\end{enumerate}
\end{proposition}

\begin{proof}
Let $A$ be any bialgebra (e.g., $B$ or $B^\cop$).
We already know from \cref{rem:cosep} that $C \coloneqq B\oslash B$ is coseparable as a coalgebra in $_{A}\mathfrak{M}$, if and only if there is a $C$-bicomodule morphism $\vartheta \colon C\otimes C\rightarrow C$ in $_{A}\mathfrak{M}$ such that $\vartheta \circ \Delta _{C}=\mathrm{Id}$ and that the existence of $\vartheta$ is equivalent to the existence of $\beta \colon C \obslash C \to \Bbbk$ (or $\gamma$) such that \eqref{oslcosep} (or \eqref{eq:oslcosepcop}) holds true.

\begin{itemize}[leftmargin=0.5cm]
\item[\ref{item:cosep1}]
%
%
If we further require that $\vartheta \circ \Delta _{C}=\mathrm{Id}$, then%
\begin{equation*}
\beta _{\vartheta }\left( \left( 1\oslash 1\right) \obslash \left( 1\oslash
1\right) \right) = \varepsilon _{C}\vartheta \Delta _{C}\left(  1\oslash
1\right)  =\varepsilon _{C}\left( 1\oslash 1\right) =1.
\end{equation*}
Conversely, if we further require that $\beta \left( \left( 1\oslash
1\right) \obslash \left( 1\oslash 1\right) \right) =1$, then
\begin{align*}
\vartheta _{\beta }\Delta _{C}(x\oslash y) &=\vartheta _{\beta }\left( \left(
x_{1}\oslash y_{2}\right) \otimes \left( x_{2}\oslash y_{1}\right) \right)
=\left( x_{1}\oslash y_{3}\right) \beta \left( \left( x_{2}\oslash
y_{2}\right) \obslash \left( x_{3}\oslash y_{1}\right) \right)  \\
&=\left( x_{1}\oslash y_{3}\right) \beta \left( x_{2}\left( 1\oslash
y_{2}\right) \obslash x_{3}\left( 1\oslash y_{1}\right) \right) \\
&= \left( x_{1}\oslash y_{3}\right) \beta \left( x_{2}\left(\left( 1\oslash
y_{2}\right) \obslash \left( 1\oslash y_{1}\right)\right) \right)  \\
&=\left( x\oslash y_{3}\right) \beta \left( \left( 1\oslash
y_{2}\right) \obslash \left( 1\oslash y_{1}\right)\right)  \\
&=\left( x\oslash y_{4}\right) \beta \left( y_{1}\left(\left( 1\oslash
y_{3}\right) \obslash \left( 1\oslash y_{2}\right)\right)\right)  \\
&=\left( x\oslash y_{5}\right) \beta \left( \left( y_{1}\oslash
y_{4}\right) \obslash \left( y_{2}\oslash y_{3}\right)\right)  \\
&=\left( x\oslash y_{3}\right) \beta \left( \left( y_{1}\oslash
y_{2}\right) \obslash \left( 1\oslash 1\right)\right)  \\
&=\left( x\oslash y\right) \beta \left( \left( 1\oslash 1\right) \obslash
\left( 1\oslash 1\right) \right) =x\oslash y.
\end{align*}

\item[\ref{item:cosep2}]
If we further require that $\vartheta \circ \Delta _{C}=\mathrm{Id}$, then%
\begin{equation*}
\gamma _{\vartheta }\left( \left( 1\oslash 1\right) \obslash \left( 1\oslash
1\right) \right) =\varepsilon _{C}\vartheta \Delta _{C}\left(  1\oslash
1\right)  =\varepsilon _{C}\left( 1\oslash 1\right) =1.
\end{equation*}
Conversely, if we further require that $\gamma \left( \left( 1\oslash
1\right) \obslash \left( 1\oslash 1\right) \right) =1$, then
\begin{align*}
\vartheta _{\gamma }\Delta _{C}(x\oslash y) &=\vartheta _{\gamma }\left( \left(
x_{1}\oslash y_{2}\right) \otimes \left( x_{2}\oslash y_{1}\right) \right)
=\left( x_{1}\oslash y_{3}\right) \gamma \left( \left( x_{2}\oslash
y_{2}\right) \obslash \left( x_{3}\oslash y_{1}\right) \right)  \\
&=\left( x_{1}\oslash y_{3}\right) \gamma \left( y_{2}\left( x_{2}\oslash
 1 \right) \obslash y_{1}\left( x_{3}\oslash 1\right) \right) \\
&= \left( x_{1}\oslash y_{2}\right) \gamma \left( y_{1}\left(\left( x_{2}\oslash
1\right) \obslash \left( x_{3}\oslash 1\right)\right) \right)  \\
&=\left( x_{1}\oslash y\right) \gamma \left( \left( x_{2}\oslash
1\right) \obslash \left( x_{3}\oslash 1\right)\right)  \\
&=\left( x_{1}\oslash y\right) \gamma \left( x_{4}\left(\left( x_{2}\oslash 1\right) \obslash \left( x_{3}\oslash 1\right)\right)\right)  \\
&=\left( x_{1}\oslash y\right) \gamma \left( \left( x_{2}\oslash
x_{5}\right) \obslash \left( x_{3}\oslash x_{4}\right)\right)  \\
&=\left( x_{1}\oslash y\right) \gamma \left( \left( x_{2}\oslash
x_{3}\right) \obslash \left( 1\oslash 1\right)\right)  \\
&=\left( x\oslash y\right) \gamma \left( \left( 1\oslash 1\right) \obslash
\left( 1\oslash 1\right) \right) =x\oslash y. \qedhere
\end{align*}
\end{itemize}%
\end{proof}

\begin{corollary}
\label{cor:BoslBcosepint}
Let $B$ be a bialgebra. Suppose that $B \oslash B$ is coseparable as a coalgebra in $_{B^\cop}\mathfrak{M}$ and consider $\gamma $ from \cref{prop:BoslashBcosep}\,\ref{item:cosep2}. Then, the map $f \colon B\oslash B\rightarrow \Bbbk $ defined by setting \[f\left( x\oslash y\right) \coloneqq \gamma (\left( x\oslash 1\right) \obslash \left( y\oslash 1\right) ) \qquad \text{for all } x,y \in B,\] is a total integral in $\pred{B}$.
\end{corollary}

\begin{proof}
    Set $f\left( x\otimes y\right) \coloneqq \gamma (\left( x\oslash 1\right)
    \obslash \left( y\oslash 1\right) )$ and let us check that $f$ is a map from $B\oslash B$.
    Indeed, for $x,y,b\in B$, we have
    \begin{align*}
    f\left( xb_{1}\otimes yb_{2}\right)  &= \gamma (\left( xb_{1}\oslash 1\right) \obslash \left( yb_{2}\oslash 1\right) ) = \varepsilon _{B}\left( b_{3}\right) \gamma (\left( xb_{1}\oslash 1\right) \obslash \left( yb_{2}\oslash 1\right) )  \\
    & = \gamma (b_{3}\left( \left( xb_{1}\oslash 1\right) \obslash \left(
    yb_{2}\oslash 1\right) \right) ) = \gamma (b_{4}\left( xb_{1}\oslash 1\right) \obslash b_{3}\left(
    yb_{2}\oslash 1\right) )  \\
    & = \gamma (\left( xb_{1}\oslash b_{4}\right) \obslash \left(
    yb_{2}\oslash b_{3}\right) ) = \gamma (\left( xb_{1}\oslash b_{3}\right) \obslash \left( y\oslash 1\right)\varepsilon_B(b_{2}) )  \\
    & = \gamma (\left( xb_{1}\oslash b_{2}\right) \obslash \left(
    y\oslash 1\right) ) = \gamma \left( \left( x\oslash 1\right) \obslash \left( y\oslash 1\right)
    \right) \varepsilon _{B}\left( b\right) =f\left( x\otimes y\right)
    \varepsilon _{B}\left( b\right) .
    \end{align*}
    Thus we can define $f \colon B\oslash B\to \Bbbk,\; x\oslash y\mapsto\gamma (\left( x\oslash 1\right)
    \obslash \left( y\oslash 1\right) )$.

    If we write \eqref{eq:oslcosepcop} on elements, for all \ $x,x^{\prime },y,y^{\prime }\in B$, we get%
    \begin{equation*}
    \left( x\oslash x^{\prime }\right) _{1}\gamma (\left( x\oslash x^{\prime
    }\right) _{2}\obslash \left( y\oslash y^{\prime }\right) )=\gamma (\left(
    x\oslash x^{\prime }\right) \obslash \left( y\oslash y^{\prime }\right)
    _{1})\left( y\oslash y^{\prime }\right) _{2}
    \end{equation*}%
    i.e.%
    \begin{equation*}
    \left( x_{1}\oslash x_{2}^{\prime }\right) \gamma (\left( x_{2}\oslash
    x_{1}^{\prime }\right) \obslash \left( y\oslash y^{\prime }\right) )=\gamma
    (\left( x\oslash x^{\prime }\right) \obslash \left( y_{1}\oslash
    y_{2}^{\prime }\right) )\left( y_{2}\oslash y_{1}^{\prime }\right) .
    \end{equation*}%
    If we let $x^{\prime }=y^{\prime }=1$ in the previous formula, then
    we obtain
    \begin{equation*}
    x_{1}\oslash 1\gamma (\left( x_{2}\oslash 1\right) \obslash \left( y\oslash
    1\right) )=\gamma (\left( x\oslash 1\right) \obslash \left( y_{1}\oslash
    1\right) )y_{2}\oslash 1
    \end{equation*} that is $x_{1}f(x_{2}\oslash y)\oslash 1 = f(x\oslash
    y_{1})y_{2}\oslash 1$  i.e.\ \eqref{form:integosl3}. Moreover $f\left(
    1\oslash 1\right) = \gamma (\left( 1\oslash 1\right) \obslash \left( 1\oslash
    1\right) )=1$.
\end{proof}

\subsection{Integrals for one-sided Hopf algebras}\label{ssec:intnHopf}

We describe integrals in $\pred{B}$ for $B$ a right Hopf algebra with anti-multiplicative and anti-comultiplicative
right antipode.

\begin{proposition}\label{pro:intgRHopf}
Given a right Hopf algebra $B$ with anti-multiplicative and anti-comultiplicative
right antipode $S^r$, the space of left integrals in $\pred{B}$ is isomorphic to the space of   linear maps $\tau:B\rightarrow \Bbbk $ such that
\begin{equation}\label{intrHopf}
x_{1}\tau\left( x_{2}S^{r}\left(
y\right) \right) =\tau\left( xS^{r}\left( y_{1}\right) \right) y_{2}.
\end{equation}
\end{proposition}

\begin{proof}
 As a consequence of \cref{prop:Frobenius}, given a right Hopf algebra $B$ with anti-multiplicative and anti-comultiplicative
right antipode, we  have the isomorphism $i_B^*=\mathfrak{M}\left( i_B,\Bbbk \right):\mathfrak{M}%
\left( B\oslash B,\Bbbk \right) \to \mathfrak{M}\left( B,\Bbbk \right),f\mapsto f\circ i_B.$
By definition, a left integral in $\pred{B}$ is a linear map $f:B\oslash B\rightarrow \Bbbk $
such that $x_{1}f\left( x_{2}\oslash y\right) \oslash 1=f\left( x\oslash
y_{1}\right) y_{2}\oslash 1$, thus via $i_B^*$  it  corresponds to a linear map $\tau_{f}\coloneqq f\circ
i_{B}:B\rightarrow \Bbbk $ such that \eqref{intrHopf} is true, knowing that $i_{B}^{-1}\left( x\oslash y\right)
=xS^{r}\left( y\right) $.
\begin{invisible}
Given, $f,$ by applying $i_{B}^{-1}$ on both sides of $x_{1}f\left(
x_{2}\oslash y\right) \oslash 1=f\left( x\oslash y_{1}\right) y_{2}\oslash 1,
$ we get $x_{1}f\left( x_{2}\oslash y\right) =f\left( x\oslash y_{1}\right)
y_{2}$ so that
\begin{eqnarray*}
x_{1}t_{f}\left( x_{2}S^{r}\left( y\right) \right)  &=&x_{1}fi_{B}\left(
x_{2}S^{r}\left( y\right) \right) =x_{1}fi_{B}i_{B}^{-1}\left( x_{2}\oslash
y\right) =x_{1}f\left( x_{2}\oslash y\right)  \\
&=&f\left( x\oslash y_{1}\right) y_{2}=fi_{B}i_{B}^{-1}\left( x\oslash
y_{1}\right) y_{2}=t_{f}\left( xS^{r}\left( y_{1}\right) \right) y_{2}.
\end{eqnarray*}

Conversely, given $t:B\rightarrow \Bbbk ,$ define $f_{t}\coloneqq t\circ i_{B}^{-1}.$
Then%
\begin{equation*}
x_{1}f_{t}\left( x_{2}\oslash y\right) =x_{1}t\left( x_{2}S^{r}\left(
y\right) \right) =t\left( xS^{r}\left( y_{1}\right) \right)
y_{2}=f_{t}\left( x\oslash y_{1}\right) y_{2}.
\end{equation*}
\end{invisible}
\end{proof}

\begin{remark}\label{rmk:intgRHopf}
Given a right Hopf algebra $B$ with anti-multiplicative and anti-comultiplicative right antipode, we know that the canonical map $i_B$ is bijective by \cref{prop:Frobenius}, and so we can consider the linear map $\omega \colon \int_l \pred{B} \to \int_lB^*,f\mapsto f\circ i_B,$ from \cref{rem:oldandnew}, which becomes injective and allows us to identify $\int_l\pred{B}$ with the space of integrals in $B^*$ that satisfy \eqref{intrHopf}.
In general, this map is not evidently surjective, but it is so in case the right antipode is indeed an antipode, as stated in the next result.

\end{remark}

\begin{corollary}\label{cor:intHopf}
For a Hopf algebra $B$ the map 
$\omega \colon \int_l\pred{B}\to\int_lB^*,f\mapsto f\circ i_B,$ is invertible.
\end{corollary}

\begin{proof}
As already mentioned, $i_B$ is invertible for $B$ a Hopf algebra.
\cref{pro:intgRHopf} shows that $i_B^*$ gives an isomorphism between $\int_l\pred{B}$ and the space of  linear maps $\tau:B\rightarrow \Bbbk $ such that \eqref{intrHopf} is true. It is well-known that since $B$ is a Hopf algebra, this is equivalent to ask that $\tau\in \int_lB^*$, cf. \cite[Lemma 5.1.4]{DNC}.
\end{proof}

\subsection{The cocommutative case}\label{ssec:cocomm}

Here we study integrals for cocommutative bialgebras. We start with a useful lemma. 

\begin{lemma}
\label{lem:intBoslBHopf}
Let $B$ be a bialgebra and let $f:B\to H$ be a bialgebra map into a Hopf algebra $H$ with $\int_{l} H^{\ast}=\Bbbk \tau$, with $\tau$ possibly zero. If $\widehat{f}:B\oslash B\to H,x\oslash y\mapsto f(x)Sf(y),$ from \eqref{eq:fhat}, is bijective, then $\int_{l} \pred{B}=\int_{l}\left( B\oslash B\right)
^{\ast }=\Bbbk \lambda $ for $\lambda\coloneqq \tau \circ \widehat{f}.$
\end{lemma}


\begin{proof}
We showed in \cite[Remark 2.7]{ArMeSa} that $\widehat{f}$ is a coalgebra map such that $\widehat{f}\circ i_{B}=f $.
The bijection $\widehat{f} \colon B\oslash B \rightarrow H$ yields a bijection $\Bbbk \tau=\int_{l}H^{\ast
} \overset{\widehat{f }^{\ast }}{\longrightarrow }\int_{l}\left( B\oslash B\right) ^{\ast } $ and hence $\int_{l}\left( B\oslash B\right)
^{\ast }=\Bbbk \widehat{f }^{\ast }\left( \tau\right) =\Bbbk \tau\widehat{f }%
=\Bbbk \lambda .$
By \cref{rem:LRints}, we get  $\int_{l} \pred{B}
\subseteq \int_{l}\left( B\oslash B\right) ^{\ast }=\Bbbk \lambda $.
To have the equality it remains to check that $\lambda\in \int_{l} \pred{B}$.
Let $a,b\in B$. By \cite[Lemma 5.1.4]{DNC}, we know that $f (a_1)\tau(f(a_2)Sf(b))=\tau(f(a)Sf(b_1))f(b_2)$ i.e. $f(a_1)\lambda(a_2\oslash b)=\lambda(a\oslash b_1)f(b_2)$ and hence $a_1\lambda(a_2\oslash b)-\lambda(a\oslash b_1)b_2\in\ker(f)=\ker(\widehat{f}\circ i_B)=\ker(i_B)$.
This means that $a_{1}\lambda \left( a_{2}\oslash b\right) \oslash 1=\lambda
\left( a\oslash b_{1}\right) b_{2}\oslash 1$  and hence $%
\lambda \in \int_{l}\pred {B}$.
\end{proof}

Then we show that, in the cocommutative case, given a total integral, we can completely describe the whole space of integrals.

\begin{proposition}\label{lem:tot}
Let $B$ be a cocommutative bialgebra. If there is a total integral $\lambda\in \int_l \pred{B}$, then $\int_l\pred{B}=\Bbbk\lambda$. Moreover $\lambda$ is the unique total integral.
\end{proposition}
\begin{proof}
If we take $y=1$ in \eqref{form:integosl1}, we get $f \conv g=fg(1\oslash 1)$ for every integral $f\in \int_l \pred{B}$ and any $g\in\pred{B}\equiv \mathfrak{M}(B\oslash B,\Bbbk)$. Thus, given $\lambda'\in \int_l\pred{B}$, we get $\lambda'=\lambda'\lambda(1\oslash 1)=\lambda' \conv \lambda=\lambda \conv \lambda'=\lambda\lambda'(1\oslash 1)$ where the second last equality holds by cocommutativity of $B$. Therefore $\int_l\pred{B}\leq\Bbbk\lambda$, whence the equality holds. If $\lambda'$ is total, the above equality tells that $\lambda'=\lambda\lambda'(1\oslash 1)=\lambda$.
\end{proof}

Finally, we show that whenever $i_B$ is surjective, the left integrals for a cocommutative bialgebra $B$ are completely determined by the integrals for its Hopf envelope $B \oslash B$.

\begin{proposition}
\label{prop:intgccom-iBsu}
 Let $B$ be a cocommutative bialgebra with $i_B$ surjective. Then left integrals in $\pred{B}$ coincide with usual left integrals in the augmented algebra $(B\oslash B)^*$.
\end{proposition}

\begin{proof}
In the stated hypotheses, it follows from \cite[Proposition 3.12]{ArMeSa} that the coalgebra $B \oslash B$ carries a Hopf algebra structure with antipode $x \oslash y \mapsto y \oslash x$ and that the canonical map $i_B \colon B \to B \oslash B$ is a bialgebra map. In fact, $B \oslash B \cong \mathrm{H}(B)$, the Hopf envelope of $B$. Since, in this case, $\widehat{\imath_B}$ from \eqref{eq:fhat} is the identity map, we conclude by \cref{lem:intBoslBHopf}.
\end{proof}

\begin{invisible}
    Personal: if $M$ is not left adjustable, then the $i_B$ of $B = \Bbbk M$ is not surjective. Maybe in this case the Hopf envelope is not $B \oslash B$.
\end{invisible}

\subsection{The commutative case}\label{ssec:commcase}


Recall from \cite[Theorem 3.16]{ArMeSa} that $B \oslash B \cong \mathrm{H}(B)$ via $\widehat{\eta_B}$ when $B$ is a commutative bialgebra. As a consequence, we can deduce that the left integrals for $B$ are the left integrals for its Hopf envelope $\mathrm{H}(B)$.

\begin{proposition}
\label{prop:intcomm}
Let $B$ be a commutative bialgebra. Then, left integrals in $\pred{B}$ coincide with usual left integrals on the Hopf algebra algebra $B \oslash B$, i.e., $\int_{l} \pred{B}=\int_{l}\left( B\oslash B\right) ^{\ast }$.
\end{proposition}

\begin{proof}
    As for the cocommutative case above, $B \oslash B$ carries a Hopf algebra structure with antipode $x \oslash y \mapsto y \oslash x$ (see also \cite[Lemma 2.5]{ArMeSa}) and the canonical map $i_B \colon B \to B \oslash B$ is a bialgebra map. We conclude again by \cref{lem:intBoslBHopf}, as before.
\end{proof}

\begin{corollary}
\label{cor:commtotal}
    Let $B$ be a commutative bialgebra. If there is a total integral $\lambda \in \int_l \pred{B}$, then $\int_l \pred{B} = \Bbbk \lambda$. Moreover, $\lambda$ is the unique total integral.
\end{corollary}

\begin{proof}
    Suppose that $\lambda \in \int_l \pred{B}$ is a total integral. In particular, it is non-zero. Since $B$ is commutative, $\int_l \pred{B} = \int_l (B \oslash B)^*$ and we know, by \cite{Sullivan}, that $\dim_\Bbbk\left(\int_l (B \oslash B)^*\right) = 1$. Whence, $\int_l \pred{B} = \int_l (B \oslash B)^* = \Bbbk \lambda$. Finally, if also $\lambda'$ is a total integral, then $\lambda' = k \lambda$ for $k \in \Bbbk$ and $1 = \lambda'(1 \oslash 1) = k\lambda(1 \oslash 1) = k$ entail that $\lambda' = \lambda$.
\end{proof}

\subsection{The finite-dimensional case}\label{ssec:findim}

In the finite-dimensional case as well, the integrals for a bialgebra are the integrals for its Hopf envelope and hence the space $\int (\pred{B})$ is one-dimensional, exactly as in the presence of an antipode.
To prove this claim, recall from
\cite[Proposition 3.10]{ArMeSa} that for any left Artinian bialgebra $B$ one has $\mathrm{H}(B) = B/\ker(i_B) \cong B\oslash B$.





\begin{proposition}
\label{coro:iBcopsu}
Let $B$ be a finite-dimensional or, more generally, a left Artinian bialgebra. Then
$\int_{l} \pred{B}=\int_{l}\left( B\oslash B\right)
^{\ast }=\Bbbk \lambda $ for $\lambda\coloneqq \tau \circ \widehat{q_B}$, where $\tau$ is a non-zero integral in $\mathrm{H}(B)$.
\end{proposition}

\begin{proof}
Since $B$ is either finite-dimensional or left Artinian, $\mathrm{H}(B)$ is finite-dimensional (see the comment below \cite[Proposition 3.9]{ArMeSa}). As a consequence$\int_l\mathrm{H}(B)^*=\Bbbk \tau$ for some non-zero integral $\tau$. By \cref{lem:intBoslBHopf}, we get that $\int_{l} \pred{B}=\int_{l}\left( B\oslash B\right)
^{\ast }=\Bbbk \lambda $ for $\lambda\coloneqq \tau \circ \widehat{q_B}.$
\end{proof}

Note that also in this case we have an analogue of \cref{cor:commtotal}.

\section{Examples}\label{sec:examples}

We are going to test our new notion of integral on four examples. The first three   are bialgebras $B$ with $i_B$ injective but not surjective, while the fourth has $i_B$ surjective but not injective.

\subsection{Polynomial ring with group-like variable} \label{ssec:poly}

Let $\Bbbk \left[ X\right]$ be the polynomial ring regarded as a bialgebra via $\Delta \left( X\right) = X\otimes X,\varepsilon \left( X\right) =1.$

\begin{proposition}\label{prop:intkX}
   For $B\coloneqq \Bbbk[X]$, we have $\int_{l} \pred{B}
=\Bbbk \lambda $ where $\lambda$ is the total integral
\[\lambda \colon B\oslash B\rightarrow \Bbbk , \qquad X^{m}\oslash X^{n}\mapsto
\delta _{m,n}.\]
In particular, $B$ is coseparable in $\mathfrak{M}_{B}$.
\end{proposition}

\begin{proof}
\begin{invisible}
    \ps{[changed on 7/7/25]}
\end{invisible}
Since the bialgebra $B=\Bbbk [X] $ is commutative, \cref{prop:intcomm} entails that 
$\int_{l} \pred{B} = \int_{l}\left( B\oslash B\right)^{\ast }$. Since $B \cong \Bbbk\mathbb{N}$, we have that $\mathrm{H}(B) \cong \Bbbk\mathbb{Z} \cong \Bbbk[X,X^{-1}]$, the Hopf algebra of Laurent polynomials, where $\eta_B:\Bbbk[X]\to \Bbbk[X,X^{-1}]$ is the canonical injection, and we have the isomorphism
\[\widehat{\eta_B} \colon B \oslash B \to \Bbbk[X,X^{-1}], \qquad X^m \oslash X^n \mapsto X^{m-n}.\]
Since $\int_{l}\Bbbk[X,X^{-1}]^* = \Bbbk \tau$ where $\tau(X^n) = \delta_{n,0}$ (see e.g.~\cite[Example 5.1.5 and Theorem 5.4.2]{DNC}), we have that $\int_{l} \pred{B}
=\Bbbk \lambda $ for $\lambda = \tau \circ \widehat{\eta_B}$, that is to say,
\[\lambda(X^m \oslash X^n) = \tau(X^{m-n}) = \delta_{m-n,0} = \delta_{m,n}\]
for all $m,n \in \mathbb{N}$, as claimed. It is clearly a total integral. Hence, since $i_B$ is injective because $B$ embeds into the Hopf algebra $\Bbbk [ X,X^{-1}]$, we conclude by \cref{cor:totiBinj}.
\end{proof}

\begin{remark}
\label{rem:iKXnotepi}
As we observed above, $i_B$ is injective for $B=\Bbbk[X]$. 
  However, it is not surjective because $X$ is a group-like which is not invertible in $B$. The latter also entails that $B$ is not a one-sided Hopf algebra.
\end{remark}

\subsection{Coordinate bialgebra of \texorpdfstring{$n \times n$}{nxn} matrices}
\begin{invisible}[\ps{Modified 05/06/25}]\end{invisible}

Let $\mathrm{M}(n)$ be the bialgebra whose underlying algebra is the algebra of polynomials $\Bbbk \left[x_{ij}\mid i,j=1,\ldots,n\right]$ over a field $\Bbbk$ of characteristic $0$, with comultiplication and counit uniquely determined by
\begin{gather*}
\Delta(x_{ij}) = \sum_{k = 1}^n x_{ik} \otimes x_{kj} \qquad \text{and} \qquad \varepsilon(x_{ij}) = \delta_{ij}
\end{gather*}
for all $i,j \in \{1,\ldots,n\}$ (also denoted by $\mathcal{O}(M_n(\Bbbk))$ - see \cite[Example 1.3.8]{Montgomery}). This is the bialgebra universally representing matrix multiplication, in the sense that $\mathrm{Alg}_\Bbbk(\mathrm{M}(n),A) \cong M_n(A)$ for every $\Bbbk$-algebra $A$ (see, for instance, \cite[\S I.4]{Kassel} for the case $n=2$ or \cite[Example 2.3.10(4)]{Pareigis} for the general case). 
For the sake of avoiding overlapping with \S\ref{ssec:poly}, we suppose that $n\geq 2$.

Let $X \coloneqq \begin{pmatrix} x_{ij} \end{pmatrix}$ be the $n \times n$ matrix with $ij^{\mathrm{th}}$ entry $x_{ij}$ and consider the Hopf algebra $\mathrm{GL}\left( n\right) :=\mathrm{M}(n)\left[ t\right] /\left\langle t\det X -1\right\rangle $ (also denoted by $\mathcal{O}(GL_n(\Bbbk))$ - see, e.g., \cite[Example 1.5.7]{Montgomery}).

\begin{proposition}
For $\mathrm{char}(\Bbbk)=0$ and $B \coloneqq \mathrm{M}(n)$ one has $B\oslash B\cong \mathrm{GL}\left( n\right)=\mathrm{H}\left( B\right).$ As a
consequence $\int_{l} \pred{B} = \Bbbk\lambda $ where $\lambda$ is a total integral and $B$ is coseparable as a coalgebra in $\mathfrak{M}_B$.
\end{proposition}

\begin{proof}
Since $B$ is commutative, 
$B \oslash B \cong \mathrm{H}(B)$ via $\widehat{\eta_B}$.
 It is well-known (see e.g.\ \cite[Example 1.5]{FarinatiGaston}) that $\mathrm{GL}(n)=\mathrm{H}(\mathrm{M}(n))$: $\det(X)$ is a group-like element in $\mathrm{M}(n)$ and $\mathrm{M}(n)[\det(X)^{-1}] \cong \mathrm{GL}(n)$ is a Hopf algebra, whence we conclude by \cite[Theorem 65]{Takeuchi} and \cite[Proposition 3.14]{ArMeSa}.
 From \cite[Theorem 5]{Sullivan} we know that there is a total integral $\tau\in \mathrm{GL}(n)^*$ if and only if $\mathrm{char}(\Bbbk)=0$.
 \begin{invisible}
     Sanity check: by Dual Maschke Theorem \cite[Theorem 2.4.6]{Montgomery}, the existence of a total integral is equivalent to cosemisimplicity and in \cite[page 26]{Montgomery} it is observed that $\mathrm{GL}(n)$ is cosemisimple. See, e.g., \cite[\S3.4]{Abe} for the case in positive characteristic.
 \end{invisible}
 Then $\int_{l} \mathrm{GL}(n)^* = \Bbbk \tau$ (see \cite[Exercise 10.1.7]{Radford-book}) and, in view of \cref{lem:intBoslBHopf}, we get that  $\int_{l} \pred{B}=\int_{l}\left( B\oslash B\right)
^{\ast }=\Bbbk \lambda $ for $\lambda\coloneqq \tau \circ \widehat{\eta_B}$. Since $\lambda(1\oslash 1)=\tau(1)=1$, we get that $\lambda$ is a total integral.
Since we noticed that $B$ embeds into the Hopf algebra $\mathrm{GL}(n)$, from \cref{lem:iB-inj} we deduce that $i_B$ is injective and hence
$B$ is coseparable as a coalgebra in $\mathfrak{M}_B$ by \cref{cor:totiBinj}.
%
\begin{invisible}
In view of the computations performed by Paolo one can guess an explicit and surprisingly simple expression for  the  total integral $\lambda\in \mathrm{GL}(2)^*$ that is \begin{equation*}
\lambda \left( x_{11}^{m}x_{12}^{n}x_{21}^{p}x_{22}^{q}t^{k}\right) =\delta
_{k,m+n}\delta _{m,q}\delta _{n,p}\left( -1\right) ^{n}\frac{m!n!}{(m+n+1)!}.%
\end{equation*}
A first inspection shows that we can first define it on $\mathrm{M}(2)[t]$ and then see that it kills the ideal $\langle t\det-1\rangle$ so that the linear map $\lambda\in\mathrm{GL}(2)^*$ is well-defined. Unfortunately the checking that it is really a total integral seems not so trivial (not finished yet). For this reason we can decide not to include the explicit expression, which on the one hand adds nothing to the discussion and on the other hand could be well-known.]
\end{invisible}
\end{proof}




\begin{remark}
In line with what we observed for $\Bbbk[X]$ in \cref{rem:iKXnotepi}, also $B=\mathrm{M}(n)$  is not a right Hopf algebra as otherwise we would have
    \[1=\varepsilon(x_{11}) = (x_{11})_1S^r((x_{11})_2)= \sum_{j=0}^n x_{1j}S^r(x_{j1})\in \langle x_{1j} \mid j=1,\ldots,n\rangle,\]
    a contradiction. Similarly it is not a left Hopf algebra.
    Since we have noticed that the map $i_B$ is injective, we are again in the setting of \cref{rem:iKXnotepi} so that $i_B$ is not surjective.
\end{remark}

\subsection{Quantum plane}

Fix $q\in\Bbbk\setminus\{0\}$ and consider from \cite[Section IV.1]{Kassel} the so-called \emph{quantum plane}
i.e. the bialgebra $\Bbbk _{q}\left[ x,y\right] $ generated by $x,y$ with
relation $yx=qxy.$ Its bialgebra structure, given explicitly in \cite[page 118]{Kassel}, is uniquely determined by
\begin{equation}\label{eq:coalgquantumplane}
\begin{gathered}
\Delta \left( x\right)  =x\otimes x,\qquad \Delta \left( y\right)
=x\otimes y+y\otimes 1 \\
\varepsilon \left( x\right)  =1,\qquad \varepsilon \left( y\right) =0.
\end{gathered}
\end{equation}
The quantum binomial identity tells us that%
\begin{equation}\label{form:deltaq}
\Delta \left( x^{m}y^{n}\right) =\sum_{i=0}^{n}\binom{n}{i}%
_{q}x^{m+i}y^{n-i}\otimes x^{m}y^{i}.
\end{equation}

\begin{invisible}
Since $\left( y\otimes 1\right) \left( x\otimes y\right) =yx\otimes
y=qxy\otimes y=q\left( x\otimes y\right) \left( y\otimes 1\right) ,$ the
quantum binomial identity tell us that $\Delta \left( y^{n}\right)
=\sum_{i=0}^{n}\binom{n}{i}_{q}\left( x\otimes y\right) ^{i}\left(
y\otimes 1\right) ^{n-i}=\sum_{i=0}^{n}\binom{n}{i}%
_{q}x^{i}y^{n-i}\otimes y^{i}$ and hence $\Delta \left( x^{m}y^{n}\right)
=\sum_{i=0}^{n}\binom{n}{i}_{q}x^{m+i}y^{n-i}\otimes x^{m}y^{i}.$
\end{invisible}

Note that this bialgebra is neither commutative nor cocommutative.


Consider also the Hopf algebra $\Bbbk_q[x,x^{-1},y]$ (see \cite[Section 2]{Chen}), whose bialgebra structure is uniquely determined by relations \eqref{eq:coalgquantumplane}, as well.

\begin{theorem}
\label{thm:intqplane}
For the bialgebra $B\coloneqq \Bbbk _{q}\left[ x,y\right] $, we have $\mathrm{H}(B) \cong \Bbbk_q[x,x^{-1},y] \cong B \oslash B$ and $\int_{l}\pred{B}\cong \int_{l}\mathrm{H}(B)^* = 0$.

\end{theorem}

\begin{proof}
Set $A \coloneqq \Bbbk[x]$ and $L \coloneqq \Bbbk[x,x^{-1}]$, and denote by $\jmath \colon A \to B$, by $\imath \colon L \to \Bbbk_q[x,x^{-1},y]$ and by $\eta_B \colon B\to \Bbbk_q[x,x^{-1},y]$ the canonical inclusions. Consider the algebra endomorphism $\varphi \colon L \to L$ uniquely determined by $\varphi(x) = qx$ (i.e., $\varphi(x^z) = q^{z}x^{z}$ for every $z \in \mathbb{Z}$) and its restriction to $A$, that we denote by $\varphi$ again. With these conventions, $\Bbbk _{q}\left[ x,y\right] = A[y,\varphi,0]$ and $\Bbbk _{q}\left[ x,x^{-1},y\right] = L[y,\varphi,0]$ are Ore extensions as in \cite[\S5.6]{DNC}. Namely, if we identify $L$ with the group algebra $\Bbbk \mathbb{Z}$, then in the notation of \cite[\S5.6]{DNC} we can take $c_1 = 1$, $c_1^*\colon \mathbb{Z} \to \Bbbk^\times, z \mapsto q^{z}$, and we find that $L[y,\varphi,0]$
is the $A_1$ of \cite{DNC} (the assumptions of algebraic closure and of characteristic zero on $\Bbbk$ do not play any role in this).

Now, suppose that $f \colon B \to H$ is a bialgebra morphism from $B$ to a Hopf algebra $H$. Since $x$ is group-like in $B$, $f(x)$ is invertible in $H$ and hence the composition $\xymatrix @C=20pt{A \ar[r]|-{\,\jmath\,} & B \ar[r]|-{\,f\,} & H}$ extends uniquely to an algebra morphism $\tilde f \colon L \to H$ mapping $x^{-1}$ to $Sf(x)$. Moreover, the element $f(y) \in H$ satisfies
\[f(y)\tilde f(x^n) = f(y)f\jmath(x^n) = f(yx^n) = f(q^nx^ny) = f\jmath\varphi(x^n)f(y) = \tilde f \varphi(x^n)f(y)\]
for every $n \in \mathbb{N}$ and so, by multiplying by $\tilde f(x^{-n}) $ from the right and by $\tilde f \varphi(x^{-n})$ from the left,
$f(y)\tilde f(x^{-n}) = \tilde f \varphi(x^{-n})f(y)$
for every $n \in \mathbb{N}$, too.
Thus, $f(y)\tilde f(a) = \tilde f \varphi (a) f(y)$ for every $a \in L$ and therefore, by the universal property of the Ore extension (see e.g.\ \cite[Lemma 5.6.4]{DNC}), there exists a unique algebra morphism $F \colon \Bbbk _{q}\left[ x,x^{-1},y\right] \to H$ extending $\tilde f$ and which is, in fact, a Hopf algebra morphism.
\begin{invisible}
Indeed, it satisfies
\[\Delta_H F(x) = \Delta_H f(x) = f(x) \otimes f(x) = (F \otimes F)\Delta(x) \qquad \varepsilon F(x) = \varepsilon f(x) = \varepsilon(x)\]
and
\[\Delta_H F(y) = \Delta_H f(y) = (f \otimes f)\Delta(y) = (F \otimes F)\Delta(y).\]
\end{invisible}%
Since any Hopf algebra morphism $\Bbbk _{q}\left[ x,x^{-1},y\right] \to H$ is, in particular, an algebra morphism, and so it is uniquely determined by the images of $x$ and $y$, it follows that $F$ is also the unique Hopf algebra morphism extending $f \colon B \to H$. Thus, $(\Bbbk _{q}\left[ x,x^{-1},y\right],\eta_B)$ satisfies the universal property of the Hopf envelope, as claimed.

Now, set $H \coloneqq \Bbbk_q[x,x^{-1},y]$ and consider the coalgebra map $\widehat{\eta_B} \coloneqq B\oslash B\to H$ induced by $\eta_B$. Define the linear map $\gamma \colon H\to B\oslash B$ on the basis $\{y^mx^n\mid m\in\mathbb{N},n\in\mathbb{Z}\}$ by setting $\gamma(y^mx^n)=y^mx^n\oslash 1$ if $n\geq 0$ and $\gamma(y^mx^n)=y^m\oslash x^{-n}$ otherwise. Then $\widehat{\eta_B}\circ\gamma=\id_H$ as on the basis we have $\widehat{\eta_B}\gamma(y^mx^n)=\widehat{\eta_B}(y^mx^n\oslash 1)=y^mx^n$ if $n\geq 0$ and $\widehat{\eta_B}\gamma(y^mx^n)=\widehat{\eta_B}(y^m\oslash x^{-n})=y^mS_H(x^{-n})=y^mx^n$ otherwise. Moreover, $q$-commutativity, the equality $ax\oslash bx=a\oslash b$ for $a,b\in B,$ and the equalities
\begin{gather}
0 =a\oslash b\varepsilon \left( y\right) =ay_{1}\oslash by_{2}=ax\oslash
by+ay\oslash b1 \qquad \Rightarrow \qquad ax\oslash by = -ay\oslash b, \label{eq:oslash1} \\
\text{and} \qquad a\oslash by =ax\oslash byx=qax\oslash bxy \stackrel{\eqref{eq:oslash1}}{=} -qay\oslash bx, \notag
\end{gather}%
from which one inductively gets
\[
a\oslash by^{n} = (-1)^{n}q^{\frac{n(n+1)}{2}}ay^{n}\oslash bx^{n},
\]%
\begin{invisible}
By induction:
\begin{eqnarray*}
ay^{n}\oslash bx^{n} & = & ay^{n-1}y\oslash bx^{n-1}x=-q^{-1}ay^{n-1}\oslash
bx^{n-1}y \\
& = & -q^{-1}q^{-(n-1)}ay^{n-1}\oslash byx^{n-1} \\
& = & -q^{-1}q^{-(n-1)}\left( -1\right) ^{n-1}q^{-\frac{\left( n-1\right) n}{2}} a\oslash byy^{n-1} \\
& = &(-1)^{n}q^{-n-\frac{\left( n-1\right) n}{2}}a\oslash by^{n} \\
& = &(-1)^{n}q^{-\frac{n(n+1)}{2}}a\oslash by^{n}.
\end{eqnarray*}
\end{invisible}%
imply that $B\oslash B$ is spanned by $\{y^mx^n\oslash 1,y^m\oslash x^n\mid m,n\in\mathbb{N}\}\subseteq \mathrm{im}(\gamma)$ so that the linearity of $\gamma$ entails that this map is surjective. Therefore $\gamma$ is bijective and hence $\widehat{\eta_B}$ is its inverse. Since the bialgebra map $\eta_B \colon B \to H$ induces the bijection $\widehat{\eta_B} \colon B \oslash B \to H$, we have that $\int_l \pred{B} = \int_l (B \oslash B)^* \cong \int_l H^*$, in view of \cref{lem:intBoslBHopf}.

To conlcude, the space $\int_{l} H^*$ of ordinary integrals for $H$ is always zero, as stated in \cite[Proposition 5.6.11(ii)]{DNC} and in view of the fact that $H = A_1$.
\begin{invisible}
Nevertheless, since we dropped the standing assumptions on the base field assumed in \cite[\S 5.6]{DNC} and for the convenience of the unaccustomed reader, let us provide a different and more elementary proof.
Let $f\in\int_{l} H^*$. By computing $h_1f(h_2)=1f(h)$ for a basis element $h=x^my^n$, $m \in \mathbb{Z}, n\in \mathbb{N}$, and by using \eqref{form:deltaq} (which is still valid for $H$, e.g.\ in view of \cite[(5.10)]{DNC} or by adapting \cite[Lemma 2.1]{Chen}), we get that
\[1f \left( x^{m}y^{n}\right) = \sum_{i=0}^{n}\binom{n}{i}_{q}x^{m+i}y^{n-i}f(x^{m}y^{i}) = \sum_{i=0}^{n-1}\binom{n}{i}_{q}x^{m+i}y^{n-i}f(x^{m}y^{i}) + x^{m+n}f(x^my^n),\]
from which we can deduce that $f(x^my^n)=0$ whenever $m+n \neq 0$, since in this case $x^{m+n}$ does not appear on the left-hand side.
Furthermore, for every $k \in \mathbb{N}\setminus\{0\}$ consider the relation
\begin{equation}\label{eq:intqplane}
0 = 1f(x^{-n}y^{n+k}) = \sum_{i=0}^{n+k}\binom{n+k}{i}_qx^{-n+i}y^{n+k-i}f(x^{-n}y^{i}) = \binom{n+k}{n}_qy^{k}f(x^{-n}y^{n})
\end{equation}
which holds for all $n \in \mathbb{N}$. 
If $f(x^{-n}y^n) \neq 0$ for some $n \in \mathbb{N}$, then necessarily $\binom{n+k}{n}_q = 0$ for every $k \geq 1$. In particular, when $q \neq 1$ and $k=1$, we get $0 = \binom{n+1}{n}_{q} = (n+1)_q = \frac{q^{n+1}-1}{q-1}$ and so $q^{n+1}=1$. Since the latter holds for $q=1$ as well, from now on we can suppose that $q^{n+1}=1$.
Now, the $q$-Pascal identity entails that
\begin{equation}\label{eq:Pmagic}
\binom{n+k+1}{n+1}_{q}=\binom{n+k}{n}_{q}+q^{n+1}\binom{n+k}{n+1}_{q}=\binom{n+k}{n+1}_{q}
\end{equation}
for every $k \geq 1$, and therefore
\begin{align*}
0 & = \binom{2n+1}{n}_{q} = \binom{2n+1}{n+1}_{q} \stackrel{\eqref{eq:Pmagic}}{=} \binom{2n}{n+1}_{q} \stackrel{\eqref{eq:Pmagic}}{=} \cdots \stackrel{\eqref{eq:Pmagic}}{=} \binom{2n-(n-1)}{n+1}_{q} = \binom{n+1}{n+1}_{q} = 1.
\end{align*}
leading to a contradiction. We conclude that $f(x^{-n}y^n) = 0$ for every $n \in \mathbb{N}$, whence $f = 0$.
\end{invisible}%
\end{proof}


\begin{remark}
The interested reader may verify that the algebra structure on $B \oslash B$ obtained by transport of structure from $\Bbbk_q[x,x^{-1},y]$ is uniquely determined by
\[(x^my^s \oslash x^ny^t) \cdot (x^py^u \oslash x^ry^v) = q^{(n+t)(p+u) - np}(x^my^sx^py^u \oslash x^ry^vx^ny^t)\]
for all $m,n,p,r,s,t,u,v \in \mathbb{N}$.
\begin{invisible}
We include here the proof:
    \begin{eqnarray*}
&&\left( x^{m}y^{s}\oslash x^{n}\underleftrightarrow{y^{t}}\right) \left(
x^{p}y^{u}\oslash x^{r}\underleftrightarrow{y^{v}}\right)  \\
&=&\left( -1\right) ^{t}q^{\frac{t\left( t+1\right) }{2}}\left(
x^{m}y^{s+t}\oslash x^{n+t}\right) \left( x^{p}y^{u+v}\oslash x^{r+v}\right)
\left( -1\right) ^{v}q^{\frac{v\left( v+1\right) }{2}} \\
&=&\left( -1\right) ^{t}q^{\frac{t\left( t+1\right) }{2}}\gamma \left(
x^{m}y^{s+t}x^{-n-t}x^{p}y^{u+v}x^{-r-v}\right) \left( -1\right) ^{v}q^{%
\frac{v\left( v+1\right) }{2}} \\
&=&\left( -1\right) ^{t}q^{\frac{t\left( t+1\right) }{2}}q^{\left(
n+t\right) \left( u+v\right) }\gamma \left(
x^{m}y^{s+t}x^{p}y^{u+v}x^{-n-t}x^{-r-v}\right) \left( -1\right) ^{v}q^{%
\frac{v\left( v+1\right) }{2}} \\
&=&\left( -1\right) ^{t}q^{\frac{t\left( t+1\right) }{2}}q^{\left(
n+t\right) \left( u+v\right) }\left( x^{m}y^{s+t}x^{p}y^{u+v}\oslash
x^{n+t}x^{r+v}\right) \left( -1\right) ^{v}q^{\frac{v\left( v+1\right) }{2}}
\\
&=&\left( -1\right) ^{t}q^{\frac{t\left( t+1\right) }{2}}q^{\left(
n+t\right) \left( u+v\right) }\left( x^{m}y^{s}\underleftrightarrow{%
y^{t}x^{p}}y^{u}\oslash x^{n}x^{r}\underleftrightarrow{x^{t}y^{v}}\right)  \\
&=&\left( -1\right) ^{t}q^{\frac{t\left( t+1\right) }{2}}q^{\left(
n+t\right) \left( u+v\right) +tp-tv}\left( x^{m}y^{s}x^{p}y^{u}y^{t}\oslash
x^{n+r}y^{v}x^{t}\right)  \\
&=&q^{\left( n+t\right) \left( u+v\right) +tp-tv}\left(
x^{m}y^{s}x^{p}y^{u}\oslash x^{n+r}y^{v}y^{t}\right)  \\
&=&q^{\left( n+t\right) \left( u+v\right) +tp-tv-nv}\left(
x^{m}y^{s}x^{p}y^{u}\oslash x^{r}y^{v}x^{n}y^{t}\right)  \\
&=&q^{nu+tu+tp}\left( x^{m}y^{s}x^{p}y^{u}\oslash
x^{r}y^{v}x^{n}y^{t}\right)
\end{eqnarray*}%
\end{invisible}
If we define
\[\sigma(x^my^s \otimes x^ny^t) \coloneqq q^{(m+s)(n+t)-mn}(x^ny^t \otimes x^my^s),\]
then $(B,\sigma)$ is a braided vector space and we can define a twisted algebra structure on $B \otimes B^\op$ by means of $\sigma$:
\[(x^my^s \otimes x^ny^t)\cdot(x^py^u \otimes x^ry^v) \coloneqq q^{(n+t)(p+u)-np}(x^my^sx^py^u \otimes x^ry^vx^ny^t)\]
extended by $\Bbbk$-linearity. The algebra $B \oslash B$ is a quotient algebra of this and not of $B\otimes B^\op$.
\end{remark}

\begin{remark}
    Also the quantum plane is neither a left nor a right Hopf algebra as otherwise $x$ would be invertible in $\Bbbk _{q}[ x,y] $, which is not the case. Since we have noticed that the map $i_B$ is injective, we are again in the setting of \cref{rem:iKXnotepi}. Hence $i_B$ is not surjective.
\end{remark}

\begin{invisible}
\ps{[Now the following is not valid any more]}
\begin{remark}
    It is well-known for a Hopf algebra $H$, that $\int_{l} H^* \neq 0$ implies $\int_{l} H^* $ one-dimensional, see e.g. \cite[Theorem 5.4.2]{DNC}. \cref{thm:intqplane} shows that a similar property needs not to be true for $\int_{l} \pred{B} $ in case $B$ is just a bialgebra.
\end{remark}
\end{invisible}

\subsection{A finite-dimensional example}\label{ssec:fdimexample}
The next example we are going to investigate is given by a finite-dimensional bialgebra $B$ with $i_B$ surjective but not injective.

To this aim we turn back to the quantum plane $\Bbbk _{q}\left[ x,y\right] ,$ with $yx=qxy$ and coalgebra
structure given by $\Delta \left( x\right) =x\otimes x,\Delta \left(
y\right) =x\otimes y+y\otimes 1.$
Now  $I=\left\langle x^{3}-x,y^{2}\right\rangle $ is a bi-ideal in case $q=-1$.

\begin{invisible}
Note that $\Delta \left( y^{2}\right) =\left( x\otimes y+y\otimes 1\right)
\left( x\otimes y+y\otimes 1\right) =x^{2}\otimes y^{2}+xy\otimes
y+yx\otimes y+y^{2}\otimes 1$

$=x^{2}\otimes y^{2}+xy\otimes y+qxy\otimes y+y^{2}\otimes 1=x^{2}\otimes
y^{2}+\left( q+1\right) xy\otimes y+y^{2}\otimes 1.$

If $q=-1$ this element belongs to $I\otimes A+A\otimes I.$ Moreover $\Delta
\left( x^{3}-x\right) =x^{3}\otimes x^{3}-x\otimes x=\left( x^{3}-x\right)
\otimes x^{3}+x\otimes \left( x^{3}-x\right) \in I\otimes A+A\otimes I.$

Moreover $\varepsilon \left( y^{2}\right) =0=\varepsilon \left(
x^{3}-x\right) .$
\end{invisible}

The example we want to study is the quotient bialgebra $\Bbbk _{-1}\left[ x,y\right]/I$ i.e.
\begin{equation*}
B=\Bbbk \left\langle x,y\mid yx=-xy,x^{3}=x,y^{2}=0\right\rangle
\end{equation*}%
which is $6$-dimensional with basis $\left\{ 1,x,x^{2},y,xy,x^{2}y\right\} .$ In \cite[Proposition 4.1 and Remark 4.2]{ArMeSa} we showed that $\mathrm{H}(B) = H_{4}=\Bbbk \left\langle x,y\mid yx=-xy,x^{2}=1,y^{2}=0\right\rangle$, the Sweedler's $4$-dimensional Hopf algebra, and that $B \oslash B \cong H_4$ via the coalgebra isomorphism $\widehat{\pi } \colon B\oslash B \rightarrow H_{4}$ induced by the bialgebra map $B\rightarrow {B}/{\left\langle x^{2}-1\right\rangle }\cong H_{4}$, $x^{m}y^{n}\mapsto x^{m}y^{n}$.\medskip

In the following, $\left[ z\right] _{2}$ denotes  the congruence class of $z\in \mathbb{Z}$ modulo $2.$


\begin{proposition}
\label{prop:quotquant}
For $B=\Bbbk \left\langle x,y\mid
yx=-xy,x^{3}=x,y^{2}=0\right\rangle $ we have
$\int_{l} \pred{B} = \int_{l}\left(
B\oslash B\right) ^{\ast } =\Bbbk \lambda $ where the $\lambda$ is defined by $\lambda \left( x^{m}y^{n}\oslash x^{s}y^{t}\right) \coloneqq \left( -1\right) ^{s+t}\delta _{n+t,1}\delta _{\left[m+s+t\right] _{2},\left[ 1\right] _{2}}$ for $0\leq m,s\leq 2$ and $0\leq n,t\leq 1.$
\end{proposition}

\begin{proof}
It is known that $\int_{l} H_4^{\ast } =\Bbbk \tau$ where $\tau \colon H_4\rightarrow \Bbbk $ is the left integral defined by $\tau\left( x^{m}y^{n}\right) =\delta _{m,1}\delta_{n,1}$ for $0\leq m,n\leq 1$.
From \cref{lem:intBoslBHopf} and \cref{coro:iBcopsu}, we know that $\int_l \pred{B} = \int_l(B \oslash B)^\ast = \Bbbk \lambda$ where $\lambda \coloneqq \tau \circ \widehat{\pi}$.
%
It remains to prove that $\lambda$ is as in the statement. We compute
\[\lambda(x^my^n\oslash x^sy^t) = \tau(x^my^nS_{H_4}(x^sy^t))=\tau(x^my^n(yx)^tx^s) = (-1)^{\frac{t(t+1)}{2}+tn + s(n+t)}\tau(x^{m+t+s}y^{n+t}).\]
Since $y^2=0$ and by definition of $\tau$, the last term is zero unless $n+t=1$ and in this case the exponent becomes $\frac{t(t+1)}{2}+tn + s(n+t) = \frac{t(t+1)}{2}+tn + s = t+tn+s = t+s$ as $n+t=1$ implies either $(n,t)=(0,1)$ or $(n,t)=(1,0)$ and in both cases we have $\frac{t(t+1)}{2}=t$ and $nt=0$. Hence we get
$\lambda(x^my^n\oslash x^sy^t)=(-1)^{s+t}\delta_{n+t,1}\tau(x^{m+s+t}y)=(-1)^{s+t}\delta_{n+t,1}\delta_{[m+s+t]_2,[1]_2}$.
\end{proof}

\subsection{Comparing old and new integrals}

Let $B$ be a bialgebra.
When $i_B \colon B \to B \oslash B$ is injective, we have the linear map
\(
\omega \colon \int_l\pred{B} \to \int_lB^*, f \mapsto f \circ i_B,
\)
from \cref{rem:oldandnew}. In some cases, this becomes a bijection even if $B$ is not a Hopf algebra, as the following examples show.

\begin{example}
\label{ex:intdif}
Let $B\coloneqq \Bbbk _{q}\left[ x,y\right] $.
The space $\int_{l} B^* = \{\lambda \in B^*\mid b_1\lambda(b_2)=1\lambda(b)\text{ for all }b\in B\} $ of ordinary integrals in $B^*$ is always zero. In fact, by computing $b_1\lambda(b_2)=1\lambda(b)$ for a generator $b=x^my^n$ and by using \eqref{form:deltaq}, we get that
\[1\lambda \left( x^{m}y^{n}\right) =\sum_{i=0}^{n}\binom{n}{i}_{q}x^{m+i}y^{n-i}\lambda(x^{m}y^{i}).\]
Since $y$ does not appear on the left-hand side, we can deduce that $1\lambda \left( x^{m}y^{n}\right) = x^{m+n}\lambda(x^{m}y^{n})$ for all $m,n \in \mathbb{N}$ and hence $\lambda(x^my^n)=0$ whenever $m+n>0$. However, $1 \lambda(y) = x\lambda(y) + y \lambda(1)$ entails that also $\lambda(1) = 0$ and hence $\lambda = 0$. In this case, \eqref{eq:intcomparison} is an isomorphism, in view of \cref{thm:intqplane}.
\end{example}

\begin{example}
    Let $B \coloneqq \Bbbk[X]$ with $X$ group-like as in \cref{ssec:poly}. In this case, $\tau \colon \Bbbk[X] \to \Bbbk, X^n \mapsto \delta_{n,0},$ is an integral in $B^*$ and it coincides with $\lambda \circ i_B$, where $\lambda$ is the integral from \cref{prop:intkX}.
\end{example}

\begin{invisible}
\ps{
    More generally, let $\varphi \colon B \to H$ be an injective morphism of bialgebras, where $H$ is a Hopf algebra (which entails, in particular, that $i_B$ is injective, in view of the final comment of \cref{ssec:iB}). If $\tau \in \int_lH^*$, then $\tau \circ \hat \varphi \in \int_l\pred{B}$. Indeed,
    \[\varphi\left(x_1\tau\hat\varphi(x_2 \oslash y)\right) = \varphi(x)_1\tau\left(\varphi(x)_2S\varphi(y)\right) = \tau\left(\varphi(x)S(\varphi(y)_1)\right)\varphi(y)_2 = \varphi\left(\tau\hat\varphi(x \oslash y_1)y_2\right)\]
    and then we conclude by injectivity of $\varphi$ and \cref{pro:cosepiB}.
    This means that we have a chain of linear maps
    \[\int_lH^* \to \int_l\pred{B} \subseteq \int_l(B \oslash B)^* \to \int_l B^*.\]
    The composition of the last two is \eqref{eq:intcomparison}. The composition of the first two is the (co)restriction of $\hat\varphi^*$ to the spaces of integrals. In particular, if $\hat\varphi$ is an isomorphism, then $\int_lH^* \cong \int_l\pred{B} \cong \int_l(B \oslash B)^*$, and if moreover $B \oslash B \cong \mathrm{H}(B)$ via $\hat\eta_B$, then $\int_l\mathrm{H}(B)^* \cong \int_l\pred{B} \cong \int_l(B \oslash B)^*$.

    \medskip

    \textbf{Conjecture:} If $\mathrm{H}(B)$ is a localisation of $B$ obtained by inverting group-like elements, as when $B$ is commutative or for $B$ the quantum plane, then integrals in $B^*$ can be extended to integrals in $\mathrm{H}(B)^*$.

    \medskip

    The idea is that inverses of group-like elements are still group-like, and integrals on multiples of group-like elements are always $0$, unless the multiple is $g^{-1}$:



    \begin{lemma}
        Let $B$ be a bialgebra and $g \in B$ be a group-like element. If $\lambda \in \int_lB^*$, then $\lambda(gb) \neq 0$ if and only if $g^{-1}$, if it exists, appears with non-zero coefficient in $b$.
    \end{lemma}

    \begin{proof}
        Since $\lambda$ is an integral, we have
        \begin{equation}\label{eq:intgrplike}
            b_1g\lambda(b_2g) = 1\lambda(bg).
        \end{equation}
        Write $\Delta(b) = \sum_i x_i \otimes y_i$, where $\{x_i\}_i$ consists of linearly independent vectors and $\{y_i\}_i$, too. Then $\{x_ig\}_i$ consists of linearly independent vectors and $\{y_ig\}_i$, too. If now we rewrite \eqref{eq:intgrplike}, we find
        \[\sum_i x_ig\lambda(y_ig) = 1 \lambda(bg).\]
        Then $\lambda(bg) \neq 0$ if and only if $1$ appears in the right-hand side sum, if and only if there exists a (necessarily unique) $i$ such that $x_ig = k1$ for some non-zero $k \in \Bbbk$.
    \end{proof}

    So, if to obtain the Hopf envelope we invert certain non-invertible group-like elements, then we can extend any integral by putting it equal to zero on the multiples of the new elements and equal to itself on any other element.

    It follows then that in these cases $\int_lB^* \cong \int_l\mathrm{H}(B)^*$.
}
\end{invisible}

However, in general this is not the case and integrals in the sense we discussed them here are essentially different from integrals in the classical sense, as the following examples show.

\begin{example}\label{ex:findimcase}
    Let $B$ be as in \cref{ssec:fdimexample}.
    We claim that $\int_lB^* = 0$. Indeed, if $\lambda \in \int_lB^*$, then $\lambda(x) = 0 =\lambda(x^2)$ as $x$ and $x^2$ are non-zero group-like elements.
    Moreover
    \begin{align*}
        x\lambda(y) + y\lambda(1) = 1 \lambda(y) & \quad\Rightarrow\quad \lambda(1) = 0 = \lambda(y), \\
        x^2\lambda(xy) + xy\lambda(x) = 1 \lambda(xy) & \quad\Rightarrow\quad \lambda(xy) = 0, \\
        x\lambda(x^2y) + x^2y\lambda(x^2) = 1 \lambda(x^2y) & \quad\Rightarrow\quad  \lambda(x^2y) = 0.
    \end{align*}
    On the other hand, we saw previously that $\dim_\Bbbk\left(\int_l\pred{B}\right)=1$.
\end{example}

In particular, classical integrals do not satisfy the uniqueness theorem, as we already anticipated in the introduction and as the next example highlights.

\begin{example}\label{ex:Atimesk}
    Let $A$ be any algebra and let $B = A \times \Bbbk$ be the bialgebra studied in \cite[Example 4.20]{ArMeSa} with component-wise algebra structure and
    \[\Delta(a,k) = (1,1) \otimes (a,0) + (a,k) \otimes (0,1) \qquad \text{and} \qquad \varepsilon(a,k) = k\]
    for all $a\in A$ and $k \in \Bbbk$. It was proved therein that $\mathrm{H}(B) \cong \Bbbk$ with canonical map $\eta_B = \varepsilon$, and that $B \oslash B \cong \Bbbk$, too. Namely, $B^+ = \{(a,0) \mid a \in A\}$ and therefore
    \[(b,x) \oslash (ca,0) + (ba,0) \oslash (0,y)
    =((b,x) \oslash (c,y))\Delta(a,0) = 0\]
    in $B \oslash B$ for all $a,b,c \in A$, $x,y \in \Bbbk$, from which, by taking $c=1,y=0$ and  $c=0,a=y=1$ respectively, it follows that
    \[(b,x) \oslash (a,0) = 0 \qquad \text{and} \qquad (b,0) \oslash (0,1) = 0\]
    for all $a,b\in B$, $x \in \Bbbk$ and so
    \[(b,x) \oslash (c,y) = (b,x) \oslash ((c,0)+(0,y)) = (b,x) \oslash (0,y) = (0,x) \oslash (0,y) = xy(1,1) \oslash (1,1)\]
    for all $(b,x),(c,y) \in B$. In particular, we have that $\hat\eta_B = \varepsilon \oslash \varepsilon \colon B \oslash B \to \Bbbk, (b,x) \oslash (c,y) \mapsto xy,$ from \eqref{eq:fhat} is an isomorphism and therefore $\int_l\pred{B} \cong \int_l(B \oslash B)^* \cong \Bbbk$ by \cref{lem:intBoslBHopf}.

    On the contrary, we claim that $\int_lB^* \cong A^*$. Indeed, pick $\lambda \in B^*$. Then
    $\lambda \in \int_lB^*$ if and only if
    \begin{equation}\label{eq:intAk}
        (1,1)\lambda(a,0) + (a,k)\lambda(0,1) = (1,1)\lambda(a,k)
    \end{equation}
    for all $a\in A,k\in\Bbbk$, if and only if $\lambda(0,k) = 0$ for all $k \in \Bbbk$: for the direct implication, take $a=0$ in \eqref{eq:intAk}; for the reverse implication, observe that \eqref{eq:intAk} is always satisfied if $\lambda(0,k) = 0$ for all $k \in \Bbbk$. Therefore, if we consider the canonical projection $p_1:A\times\Bbbk \to A$ and injection $i_1:A\to A\times\Bbbk$, then we get $\lambda\circ i_1\circ p_1=\lambda$ and we have well-defined linear maps
    \[
    \xymatrix @R=0pt{
    \int_lB^* \ar@<+0.5ex>[r]^{\Phi} & A^* \ar@<+0.5ex>[l]^{\Psi} \\
    \lambda \ar@{|->}[r] & \lambda \circ i_1 \\
    f\circ p_1 & f \ar@{|->}[l]
    }
    \]
     such that
    \[\Phi\Psi(f) = \Phi(f\circ p_1)=f\circ p_1\circ i_1 = f \qquad \text{and} \qquad \Psi\Phi(\lambda) = \Psi(\lambda\circ i_1) = \lambda\circ i_1\circ p_1 = \lambda.\]
    Hence, they are mutually inverse isomorphisms.
\end{example}

It would be interesting to find a bialgebra $B$ for which $\int_l \pred{B}$ does not coincide with $\int_l (B \oslash B)^*$.

\appendix

\section{Explicit form of the right adjoint}\label{ssec:RightAdj}



Even though we did not need  it to define integrals on arbitrary bialgebras, in this section we present an explicit description of the right adjoint $R$ to the functor $L$ described in \cref{prop:predBalg}, together with a more elementary and direct construction of the latter.

\subsection{Constructing the right adjoint}

We start by observing that the canonical coalgebra morphism $i_B \colon B \to B \oslash B$ induces an algebra morphism $\chi\coloneqq{i_B}^* \colon (B \oslash B)^* \to B^*$ which allows us to construct the composition of functors
\begin{equation}
\label{eq:leftadj}
\xymatrix{\mathfrak{M}^B\ar[r]&\prescript{}{B^*}{\mathfrak{M}}\ar[r]^-{\chi_*}& \prescript{}{(B \oslash B)^*}{\mathfrak{M}} \cong \mathfrak{M}_{\prescript{\star}{}{B}}},
\end{equation}
where $\chi_*$ denotes the \emph{restriction of scalars functor}. For the sake of simplicity, we denote by $\bla$ the $(B \oslash B)^*$-action on a left $B^*$-module obtained by restriction of scalars along $\chi$.
Note that $\chi_*$ has the right adjoint \[\chi^!:\prescript{}{(B \oslash B)^*}{\mathfrak{M}}\to \prescript{}{B^*}{\mathfrak{M}},\quad M\mapsto  \prescript{}{(B \oslash B)^*}{\mathfrak{M}}\left(B^*,M\right),\] i.e., the \emph{coinduction functor}. Here
the left $(B \oslash B)^*$-module structure on $B^*$ is induced by $\chi$ and $\prescript{}{(B \oslash B)^*}{\mathfrak{M}}\left(B^*,M\right)$ is a left $B^*$-module with respect to
\begin{equation}\label{eq:harpoon1}
(\varphi \rightharpoonup f)(\psi) \coloneqq f(\psi * \varphi)
\end{equation}
for all $\varphi,\psi \in B^*$, $f \in \prescript{}{(B \oslash B)^*}{\mathfrak{M}}\left(B^*,M\right)$.
Thus, the composition \eqref{eq:leftadj} admits the right adjoint
\[M \mapsto \chi^!(M)^\rat = \prescript{}{(B \oslash B)^*}{\mathfrak{M}}\left(B^*,M\right)^\rat.\]
We claim that these ideas can be mimicked to construct the right adjoint $R$.\medskip

First of all, recall from \cref{rmk:BotBact} that $B \oslash B$ is a left $B \otimes B^\cop$-module coalgebra. As a consequence, its linear dual $(B \oslash B)^*$ becomes a right $B \otimes B^\cop$-module with respect to
\begin{equation}\label{eq:triangles}
\left(\alpha \triangleleft b^\cop\right)(x \oslash y) \coloneqq \alpha(x \oslash by) \qquad \text{and} \qquad \left(\alpha \bra b\right)(x \oslash y) \coloneqq \alpha(bx \oslash y),
\end{equation}
for all $\alpha \in (B \oslash B)^*$, $b,x,y \in B$, which moreover satisfy
\[
\big(b_1 \rightharpoonup \chi(\alpha \triangleleft {b_2}^\cop)\big)(x) = \alpha(x b_1 \oslash b_2) = \alpha(x \oslash 1)\varepsilon(b) = \big(\varepsilon(b)\chi(\alpha)\big)(x)
\]
for all $x \in B$, i.e.
\begin{equation}\label{eq:alphactions}
b_1 \rightharpoonup \chi(\alpha \triangleleft {b_2}^\cop) = \varepsilon(b)\chi(\alpha)
\end{equation}
for all $\alpha \in (B \oslash B)^*$, $b \in B$. In addition, $(B \oslash B)^*$ becomes a right $B^\cop$-module algebra with respect to the ordinary convolution product:
\[(\alpha * \beta)(x \oslash y) = \alpha(x_1 \oslash y_2)\beta(x_2 \oslash y_1)\]
for all $\alpha,\beta \in (B \oslash B)^*$, $x,y \in B$.
In fact, as we know,
\begin{align*}
\left((\alpha * \beta)\triangleleft b^\cop\right)(x \oslash y) & \stackrel{\eqref{eq:triangles}}{=} (\alpha * \beta)(x \oslash by) = \alpha(x_1 \oslash b_2y_2)\beta(x_2 \oslash b_1y_1) \\
& \stackrel{\eqref{eq:triangles}}{=} \left((\alpha\triangleleft {b_2}^\cop) * (\beta\triangleleft {b_1}^\cop)\right)(x \oslash y) = \left((\alpha\triangleleft {b^\cop}_1) * (\beta\triangleleft {b^\cop}_2)\right)(x \oslash y)
\end{align*}
These remarks allow us to consider the category $\prescript{}{(B \oslash B)^*}{\left(\mathfrak{M}_{B^\cop}\right)}$, which we know is isomorphic to $\left(\mathfrak{M}_B\right)_{\pred{B}} \cong \mathfrak{M}_{B \ltimes \pred{B}}$ (see \cref{lem:MBA} and \cref{rmk:semi-direct}). \medskip

Secondly, recall the following construction from \cite[Proposition 2.2]{Tambara} (see also \cite[Theorem 7]{CMZ}).
Suppose that $A,B$ are algebras and that $\psi \colon B \otimes A \to A \otimes B, b \otimes a \mapsto a_\psi \otimes b^\psi$ satisfies the conditions of \cite[Proposition 2.2(ii)]{Tambara}, i.e.,
\begin{equation}\label{eq:tambara}
\begin{aligned}
 1_\psi \otimes b^\psi & = 1 \otimes b, & a_\psi \otimes 1^\psi & = a \otimes 1, \\
 (aa')_\psi \otimes b^\psi & = a_\psi a'_\Psi \otimes b^{\psi \Psi}, &  a_\psi \otimes (bb')^\psi & = a_{\Psi\psi} \otimes b^\psi {b'}^\Psi,
 \end{aligned}
 \end{equation}
 for all $a,a'\in A$, $b,b'\in B$. Then, we can consider the smash product $A \#_\psi B$, which is the vector space $A \otimes B$ together with the multiplication
 \[(a\#_\psi b)(a'\#_\psi b') = aa'_\psi \#_\psi b^\psi b'\]
 and the unit $1 \#_\psi 1$.

 \begin{remark}\label{rem:entwiningmods}
     By borrowing the notation from entwining structure theory, one can consider the category category ${}_{A}\mathfrak{M}{}_{B}(\psi)$ whose objects are left $A$-modules and right $B$ modules $M$ obeying, for every $a,\in A,b\in B,m\in M,$ the compatibility condition
    \begin{equation}
    \label{prentwinig}
    a(mc)=(a_\psi m) c^\psi.
    \end{equation}
    A morphism is just a left $A$-linear and right $B$-linear map and it can be easily verified that ${}_{A}\mathfrak{M}{}_{B}(\psi) \cong \mathfrak{M}_{A^\op \#_\psi B}$.
    Indeed, if $M$ is an object in ${}_{A}\mathfrak{M}_B(\psi)$, then we can define
    $m\cdot (a \#_\psi x) \coloneqq (am)x$.
    In this case, $m \cdot (1 \#_\psi 1) = m$ and
    \begin{align*}
    \left(m \cdot (a \#_\psi x)\right) \cdot (b \#_\psi y) & = (b((am)x))y = ((b_\psi(am))x^\psi)y = ((b_\psi a)m)(x^\psi y) \\
    & = m \cdot (b_\psi a \#_\psi x^\psi y) = m \cdot \big((a \#_\psi x)(b \#_\psi y)\big).
    \end{align*}
    Vice versa, if $N$ is an object in $\mathfrak{M}{}_{{A}^\op \#_\psi B}$, then we can define
    $(am)x \coloneqq m \cdot (a \#_\psi x)$
    and so
    \begin{align*}
    a(mx) = mx \cdot (a \#_\psi 1) = m \cdot (1 \#_\psi x) \cdot (a \#_\psi 1) = m \cdot (a_\psi \#_\psi x^\psi) = (a_\psi m)x^\psi.
    \end{align*}
    It is also clear that, with these definitions, a morphism in ${}_{A}\mathfrak{M}_B(\psi)$ is right $A^\op \#_\psi B$-linear, and vice versa.
 \end{remark}

 If $B$ is a bialgebra and $A$ is a left $B$-module algebra, then
 \begin{equation}
 \label{def:psi}
  \psi \colon B \otimes A^\op \to A^\op \otimes B,  \qquad x \otimes a \mapsto x_2 \cdot a \otimes x_1,
 \end{equation}
 satisfies conditions \eqref{eq:tambara} since
\begin{align*}
\psi(x \otimes 1) & = x_2 \cdot 1 \otimes x_1 = 1 \otimes x, \\
\psi(1 \otimes a) & = 1 \cdot a \otimes 1 = a \otimes 1, \\
\psi(x \otimes a \cdot^\op b) & = x_2 \cdot (ba) \otimes x_1 = (x_2 \cdot b)(x_3 \cdot a) \otimes x_1 = (x_3 \cdot a) \cdot^\op (x_2 \cdot b) \otimes x_1 \\
& = (x_2 \cdot a) \cdot^\op b_\Psi \otimes {x_1}^\Psi = a_\psi \cdot^\op b_\Psi \otimes {x^\psi}^\Psi, \\
\psi(xy \otimes a) & = (xy)_2 \cdot a \otimes (xy)_1 = x_2 \cdot (y_2 \cdot a) \otimes x_1y_1 = x_2 \cdot a_\Psi \otimes x_2y^\Psi = a_{\Psi\psi} \otimes x^\psi y^\Psi.
\end{align*}
In particular, for $A = B^*$ and
\begin{equation}
\label{def:psi2}
\psi \colon B \otimes {B^*}^\op \to {B^*}^\op \otimes B,\quad x \otimes \varphi \mapsto \left(x_2 \rightharpoonup \varphi\right) \otimes x_1,
\end{equation} we can consider the smash product ${B^*}^{\op} \#_\psi B$ and the categories ${}_{B^*}\mathfrak{M}{}_B(\psi) \cong \mathfrak{M}_{{B^*}^\op \#_\psi B}$.\medskip

Next aim is to show that the functors in \eqref{eq:leftadj} induce the following ones
\begin{equation}
\label{eq:leftadj2}
\xymatrix{
\mathfrak{M}^B_B \ar[r]^-{\hat{\mathcal{L}}} & {}_{B^*}\mathfrak{M}{}_B(\psi) \cong\mathfrak{M}_{{B^*}^\op \#_\psi B} \ar[r]^-{\tilde{\mathcal{L}}} & \mathfrak{M}_{B \ltimes \pred{B}}
}
\end{equation}
where the functors $\hat{\mathcal{L}}$ and $\tilde{\mathcal{L}}$ will be constructed in the following \cref{pro:Lhat} and \ref{pro:Ltilde} together with their adjoints.

\begin{proposition}
\label{pro:Lhat}
The usual functor $\mathfrak{M}^B\to{_{B^*}}\mathfrak{M}$ induces a functor \[\hat{\mathcal{L}}:\mathfrak{M}^B_B\to {}_{B^*}\mathfrak{M}_B(\psi)\] and the rational part functor $(-)^\rat:{}_{B^*}\mathfrak{M}\to \mathfrak{M}^B$ induces a functor
\[\hat{\mathcal{R}}\coloneqq(-)^\rat:{}_{B^*}\mathfrak{M}_B(\psi)\to \mathfrak{M}^B_B\]
such that $\hat{\mathcal{L}}\dashv\hat{\mathcal{R}}$.
\end{proposition}

\begin{proof}
We first prove that the usual functor $\mathfrak{M}^B\to{_{B^*}}\mathfrak{M}$ induces a functor $\hat{\mathcal{L}}.$ Indeed given $M$ in $\mathfrak{M}^B_B$, for $\varphi\in B^*,m\in M,b\in B,$ we get
\[\varphi\cdot (m \cdot b) = (m \cdot b)_0\varphi((m \cdot b)_1) = m_0 \cdot b_1\varphi(m_1b_2) = m_0 \cdot b_1(b_2\rightharpoonup\varphi)(m_1)=((b_2\rightharpoonup\varphi)\cdot m) \cdot b_1\] so that $M$ belongs to $_{B^*}\mathfrak{M}_B(\psi)$. Given $f \colon M\to N$ in $\mathfrak{M}^B_B$, it is clear that $f$ is both left $B^*$-linear and right $B$-linear, whence a morphism in ${}_{B^*}\mathfrak{M}_B(\psi).$

Next, we show that the rational part functor $(-)^\rat:{}_{B^*}\mathfrak{M}\to \mathfrak{M}^B$ induces a functor $\hat{\mathcal{R}}$ as in the statement.
Indeed, given $M$ in ${}_{B^*}\mathfrak{M}_B(\psi)$, $m\in M^\rat,b\in B$, the equality
\begin{equation}\label{eq:Mratsubmod}
\varphi \cdot (m \cdot b) \stackrel{\eqref{prentwinig}}{=} ((b_2 \rightharpoonup \varphi)\cdot m) \cdot b_1 \stackrel{(m\in M^\rat)}{=} (m_0(b_2 \rightharpoonup \varphi)(m_1)) \cdot b_1 = m_0 \cdot b_1\varphi(m_1b_2)
\end{equation}
entails that $m \cdot b\in M^\rat$ and that $\rho(m \cdot b)=m_0 \cdot b_1\otimes m_1b_2$, guaranteeing that we get a right $B$-Hopf module. Given a morphism $f \colon M\to N$ in ${}_{B^*}\mathfrak{M}_B(\psi)$, it is in particular left $B^*$-linear so that we can consider the linear map $f^\rat \colon M^\rat\to N^\rat,\,m\mapsto f(m)$. For $m\in M^\rat,b\in B,$ we have $f^\rat(m \cdot b)=f(m \cdot b)=f(m) \cdot b=f^\rat(m) \cdot b$ so that $f^\rat$ is in $\mathfrak{M}^B_B$.

The counit of the starting adjunction $\mathfrak{M}^B \leftrightarrows {}_{B^*}{\mathfrak{M}}$ is the inclusion $M^\rat\to M$ for $M\in {}_{B^*}\mathfrak{M}$, which is obviously in ${}_{B^*}\mathfrak{M}_B(\psi)$ if so is $M$, because $M^\rat$ is also a right $B$-submodule of $M$: for every $m \in M^\rat$ and $b \in B$, \cref{eq:Mratsubmod} entails that $m \cdot b$ is still in $M^\rat$. The unit is $\id_N:N\to N^\rat$ for $N\in\mathfrak{M}^B$, and it is obviously in $\mathfrak{M}^B_B$ if so is $N$. As a consequence  $\hat{\mathcal{L}}\dashv\hat{\mathcal{R}}$.
\end{proof}

In order to construct $\tilde{\mathcal{L}}$, recall that $B \ltimes \pred{B}$ is an algebra with respect to
\[1 \ltimes (\varepsilon \oslash \varepsilon) \qquad \text{and} \qquad (a \ltimes f)(b \ltimes g) = ab_1 \ltimes (f \triangleleft b_2)*^\op g,\]
and that ${B^*}^\op \#_\psi B$ is an algebra with respect to
\[\varepsilon \#_\psi 1 \qquad \text{and} \qquad (\varphi \#_\psi a)(\phi \#_\psi b) = \left(\varphi *^\op (a_2 \rightharpoonup \phi)\right) \#_\psi a_1b.\]
Consider the linear map
\[\xi \colon B \ltimes \pred{B} \to {B^*}^\op \#_\psi B, \qquad b \ltimes f \mapsto \left(b_2 \rightharpoonup \chi(f)\right) \#_\psi b_1,\]
obtained as the composition
\[\xymatrix{B\otimes\pred{B}\ar[r]^{B\otimes\chi^\op}&B\otimes B^{*\op}\ar[r]^{\psi}&B^{*\op}\otimes B}.\]

\begin{proposition}
\label{pro:Ltilde}
The morphism $\xi \colon B \ltimes \pred{B} \to {B^*}^\op \#_\psi B$ is a morphism of algebras and so the restriction of scalars along $\xi$,
\[\tilde{\mathcal{L}} \coloneqq \left(
\xymatrix{
\mathfrak{M}_{{B^*}^\op \#_\psi B} \ar[r]^-{\xi_*} & \mathfrak{M}_{B \ltimes \pred{B}}
}\right),\]
admits the right adjoint $\tilde{\mathcal{R}} \coloneqq \xi^!= \mathfrak{M}_{B \ltimes \pred{B}}\left({B^*}^\op \#_\psi B, -\right)$, the coinduction functor.
\end{proposition}

\begin{proof}
We verify that $\xi$ is a morphism of algebras. Clearly, $\xi(1 \ltimes (\varepsilon \oslash \varepsilon)) = \varepsilon \#_\psi 1$. Moreover,
\begin{align*}
    \xi(ab_1 & \ltimes (f \triangleleft b_2)*^\op g) = ((ab_1)_2 \rightharpoonup \chi((f \triangleleft b_2)*^\op g)) \#_\psi (ab_1)_1 = (a_2b_2 \rightharpoonup (\chi(g)*\chi(f \triangleleft b_3))) \#_\psi a_1b_1 \\
    & = (a_2b_2 \rightharpoonup \chi(g))*(a_3b_3 \rightharpoonup\chi(f \triangleleft b_4)) \#_\psi a_1b_1 \stackrel{\eqref{eq:alphactions}}{=} (a_2b_2 \rightharpoonup \chi(g))*(a_3 \rightharpoonup\chi(f )) \#_\psi a_1b_1 \\
    & = (a_3 \rightharpoonup\chi(f )) *^\op (a_2 \rightharpoonup (b_2 \rightharpoonup \chi(g))) \#_\psi a_1b_1 = (a_2 \rightharpoonup\chi(f ) \#_\psi a_1)(b_2 \rightharpoonup \chi(g) \#_\psi b_1) \\
    & = \xi(a \ltimes f)\xi(b \ltimes g)
\end{align*}
for all $a,b \in B$, $f,g \in \pred{B}$. Therefore, we can consider the functor given by the restriction of scalars along $\xi$
\[\xi_* \colon \mathfrak{M}_{{B^*}^\op \#_\psi B} \to \mathfrak{M}_{B \ltimes \pred{B}}.\]
It is well-known that it admits the right adjoint $\mathfrak{M}_{B \ltimes \pred{B}}\left({B^*}^\op \#_\psi B, -\right)$, as claimed.
\end{proof}

\begin{theorem}
\label{thm:Ltildehat}
The functor $L$ from \eqref{eq:functorL} coincides with the composition
\begin{equation}\label{eq:Lcomps}
\xymatrix{
\mathfrak{M}^B_B \ar[r]^-{\hat{\mathcal{L}}} & {}_{B^*}\mathfrak{M}{}_B(\psi) \cong \mathfrak{M}_{{B^*}^\op \#_\psi B} \ar[r]^-{\tilde{\mathcal{L}}} & \mathfrak{M}_{B \ltimes \pred{B}} \cong \left(\mathfrak{M}_B\right)_{\pred{B}}
}
\end{equation}
and, as such, it admits the right adjoint $R$
\begin{equation}\label{eq:Rcomps}
\xymatrix{
\left(\mathfrak{M}_B\right)_{\pred{B}} \cong \mathfrak{M}_{B \ltimes \pred{B}} \ar[r]^-{\tilde{\mathcal{R}}} & \mathfrak{M}_{{B^*}^\op \#_\psi B} \cong {}_{B^*}\mathfrak{M}{}_B(\psi) \ar[r]^-{\hat{\mathcal{R}}} & \mathfrak{M}^B_B
}
\end{equation}
In particular, the latter is given on objects by $R(M) = \mathfrak{M}_{B \ltimes \pred{B}}\left({B^*}^\op \#_\psi B, M\right)^\rat.$
\end{theorem}

\begin{proof}
Let $M$ be a right $B$-Hopf module, with right $B$-action denoted by $\cdot$ and right $B$-coaction denoted by $\delta(m) = m_0 \otimes m_1$ for all $m \in M$. Then $\hat{\mathcal{L}}(M)$ is $M$ itself with the same right $B$-module structure and left $B^*$-module structure given by $\varphi \cdot m = m_0\varphi(m_1)$ for all $m \in M$, $\varphi \in B^*$. This is seen as an object in $\mathfrak{M}_{{B^*}^\op \#_\psi B}$ via
\[m \cdot (\varphi \#_\psi b) = \left(\varphi \cdot m\right) \cdot b = m_0\cdot b\varphi(m_1),\]
for all $m \in M$, $\varphi \in B^*$ and $b \in B$, as in \cref{rem:entwiningmods}. After further applying $\tilde{\mathcal{L}}$, we land on $M$ again with module structure
\[m \cdot (b \ltimes f) = \left((b_2 \rightharpoonup \chi(f))\cdot m\right)\cdot b_1\]
for all $m \in M$, $f \in \pred{B}$ and $b \in B$, which can be seen as an object in $\left(\mathfrak{M}_B\right)_{\pred{B}}$ with respect to the right $B$-module structure
\[\mu(m \otimes b) = m \cdot(b \ltimes (\varepsilon \oslash \varepsilon)) = m \cdot b,\]
for all $m \in M$ and $b \in B$, i.e.\ the original one, and right $\pred{B}$-module structure
\[\nu(m \otimes f) = m \cdot (1 \ltimes f) = \chi(f)\cdot m = m_0\chi(f)(m_1) = m_0f(m_1 \oslash 1) = m \triangleleft f\]
for all $m \in M$ and $f \in \pred{B}$, i.e. the action $\mu_\rho$ from \cref{prop:predBalg}. Thus, the composition \eqref{eq:Lcomps} acts as $L$ on objects and so, since both act as the identity on morphisms, it coincides with $L$.

By \cref{pro:Lhat} and \cref{pro:Ltilde}, the composition \eqref{eq:Rcomps} provides a right adjoint to $L$, sending any object $M$ in $\left(\mathfrak{M}_B\right)_{\pred{B}}$ to the right Hopf module $\mathfrak{M}_{B \ltimes \pred{B}}\left({B^*}^\op \#_\psi B, M\right)^\rat$ and acting as the identity on morphisms.
\end{proof}

\subsection{Recovering the integrals from the explicit right adjoint}

 For completeness, we now directly check that the coinvariants of $R(M)$ yield the integrals we introduced. This follows from \cref{prop:RratcoB} below and the definition of integrals. \medskip

 For $M$ in $(\mathfrak{M}_B)_{\pred{B}} \cong \mathfrak{M}_{B \ltimes \pred{B}}$, consider
 \[\mathcal{R}(M) \coloneqq \left\{ \alpha \in \mathfrak{M}_{\pred{B}}(B^*,M) \mid \alpha(\chi(f)_\psi * \varphi)\cdot b^\psi = \alpha(\varphi)\cdot (b \ltimes f) \text{ for all } f \in \pred{B},\varphi \in B^*, b\in B \right\},\]
 where $M$ is a right $\pred{B}$-module by restriction of scalars along $\pred{B} \to B \ltimes \pred{B}, f \mapsto 1 \ltimes f$, the right $B$-module structure $\cdot$ on $M$ is via $B \to B \ltimes \pred{B}, b \mapsto b \ltimes (\varepsilon \oslash \varepsilon)$, and $B^*$ is a right $\pred{B}$-module via $\chi$: $\varphi \cdot f = \chi(f) * \varphi$ for all $\varphi \in B^*$, $f \in \pred{B}$.

 \begin{lemma}
     We have an isomorphism
     \[\mathfrak{M}_{B \ltimes \pred{B}}\left({B^*}^\op \#_\psi B, M\right) \xlongrightarrow{\cong} \mathcal{R}(M), \qquad F \mapsto F(- \#_\psi 1).\]
 \end{lemma}

 \begin{proof}
    Notice that for every $F \in \mathfrak{M}_{B \ltimes \pred{B}}\left({B^*}^\op \#_\psi B, M\right)$ we have
    \[F(\varphi \#_\psi b) = F((\varphi \#_\psi 1)(\varepsilon \#_\psi b)) = F((\varphi \#_\psi 1)\xi(b \ltimes \varepsilon \oslash \varepsilon)) = F(\varphi \#_\psi 1)\left(b \ltimes (\varepsilon \oslash \varepsilon)\right).\]
    Thus, we have an injection $\mathfrak{M}_{B \ltimes \pred{B}}\left({B^*}^\op \#_\psi B, M\right)\to \mathfrak{M}_{ \pred{B}}\left(B^*, M\right),\,F\mapsto F(- \#_\psi 1)$. In one direction, if $F \in \mathfrak{M}_{B \ltimes \pred{B}}\left({B^*}^\op \#_\psi B, M\right)$, then
    \begin{align*}
        F(\chi(f)_\psi * \varphi \#_\psi 1)\cdot b^\psi & = F(\varphi * ^\op \chi(f)_\psi \#_\psi 1)\cdot \left(b^\psi \ltimes (\varepsilon \oslash \varepsilon) \right) \\
        & = F\left((\varphi * ^\op \chi(f)_\psi \#_\psi 1)\xi \left(b^\psi \ltimes (\varepsilon \oslash \varepsilon) \right)\right) \\
        & = F\left((\varphi * ^\op \chi(f)_\psi \#_\psi 1) \left(\varepsilon \#_\psi b^\psi \right)\right) \\
        & = F\left(\varphi * ^\op \chi(f)_\psi \#_\psi b^\psi \right)  = F\left(\varphi * ^\op \left(b_2 \rightharpoonup \chi(f)\right) \#_\psi b_1 \right) \\
        & = F\left((\varphi \#_\psi 1)\xi(b\ltimes f)\right) = F\left(\varphi \#_\psi 1\right) \cdot (b\ltimes f)
    \end{align*}
    for all $b \in B$, $f \in \pred{B}$, $\varphi \in B^*$, so that the injection $\mathfrak{M}_{B \ltimes \pred{B}}\left({B^*}^\op \#_\psi B, M\right)\to \mathfrak{M}_{ \pred{B}}\left(B^*, M\right)$ induces an injection $\mathfrak{M}_{B \ltimes \pred{B}}\left({B^*}^\op \#_\psi B, M\right)\to \mathcal{R}(M)$. In the opposite direction, if $\alpha \in \mathcal{R}(M)$, then we can define $F_\alpha \colon {B^*}^\op \#_\psi B \to M, \varphi \#_\psi b \mapsto \alpha(\varphi)\cdot b$, which satisfies
    \begin{align*}
        F_\alpha((\varphi \#_\psi b)\xi(a \ltimes f)) & = F_\alpha((\varphi \#_\psi b)((a_2 \rightharpoonup \chi(f)) \#_\psi a_1)) = F_\alpha((\varphi *^\op (b_2 \rightharpoonup (a_2 \rightharpoonup \chi(f)))) \#_\psi b_1a_1) \\
        & = \alpha\left((b_2a_2 \rightharpoonup \chi(f))*\varphi\right)\cdot b_1a_1 = \alpha\left(\chi(f)_\psi*\varphi\right)\cdot (ba)^\psi \stackrel{(*)}{=} \alpha(\varphi) \cdot (ba \ltimes f) \\
        & = \alpha(\varphi) \cdot (b \ltimes (\varepsilon\oslash\varepsilon))(a \ltimes f) = (\alpha(\varphi) \cdot b) \cdot (a \ltimes f) = F_\alpha(\varphi \#_\psi b) \cdot (a \ltimes f)
    \end{align*}
    where in $(*)$ we used the fact that $\alpha \in \mathcal{R}(M)$. Therefore, we have an assignment $\mathcal{R}(M) \to \mathfrak{M}_{B \ltimes \pred{B}}\left({B^*}^\op \#_\psi B, M\right), \alpha \mapsto F_\alpha$, which provides an inverse for the injection above.
 \end{proof}

 \begin{remark}\label{rem:predBlinear}
     For every $\alpha \colon B^* \to M$ linear for which $\alpha(\chi(f)_\psi * \varphi)\cdot b^\psi = \alpha(\varphi)\cdot (b \ltimes f) \text{ for all } f \in \pred{B},\varphi \in B^*, b\in B$, we have
    \[\alpha(\chi(f)*\varphi) = \alpha(\chi(f)_\psi*\varphi)\cdot 1^\psi = f \cdot \alpha(\varphi),\]
    hence $\alpha$ is left $\pred{B}$-linear and so this condition can be omitted from the description of $\mathcal{R}(M)$.
 \end{remark}

\begin{proposition}\label{prop:RratcoB}
For $M$ in $(\mathfrak{M}_{B})_{\pred{B}} \cong \mathfrak{M}_{B \ltimes \pred{B}}$, the map ${\mathcal{R}(M)^\rat}^{\mathrm{co}B} \to M^{\pred{B}},\; \alpha \mapsto \alpha(\varepsilon),$ is a $\Bbbk$-linear isomorphism.
\end{proposition}

\begin{proof}
Since the left $B^*$-module structure $\mu$ on $\mathcal{R}(M)$ comes from the fact that it is isomorphic to $\mathfrak{M}_{B \ltimes \pred{B}}\left({B^*}^\op \#_\psi B, M\right)$, which is a right ${B^*}^\op \#_\psi B$-module via $\leftharpoonup$, it is of the form
\[\mu(\varphi \otimes \alpha) = \left(F_\alpha \leftharpoonup (\varphi \#_\psi 1)\right)( - \#_\psi 1) = F_\alpha((\varphi \#_\psi 1)(- \#_\psi 1)) = F_\alpha(-*\varphi \#_\psi 1) = \alpha(-*\varphi),\]
i.e., $\mu(\varphi \otimes \alpha) = \varphi \rightharpoonup \alpha$ for all $\varphi \in B^*$ and $\alpha \in \mathcal{R}(M)$. Therefore, if $\alpha \in {\mathcal{R}(M)^\rat}^{\mathrm{co}B}$ then
\begin{equation}\label{eq:ratcoinv}
\alpha(\varphi) = \alpha(\varepsilon * \varphi) = (\varphi \rightharpoonup \alpha)(\varepsilon) \stackrel{(\alpha \in \mathcal{R}(M)^\rat)}{=} \alpha_0(\varepsilon)\varphi(\alpha_1) \stackrel{(\alpha \in {\mathcal{R}(M)^\rat}^{\mathrm{co}B})}{=} \alpha(\varepsilon)\varphi(1_B)
\end{equation}
for all $\varphi \in B^*$. Therefore, $\alpha$ is uniquely determined by $\alpha(\varepsilon) \in M$, which satisfies
\begin{align*}
(\alpha(\varepsilon) \cdot b) \cdot f & = \alpha(\varepsilon) \cdot (b \ltimes f) \stackrel{(*)}{=} \alpha\left((b_2 \rightharpoonup \chi(f))*\varepsilon\right)\cdot b_1 \stackrel{\eqref{eq:ratcoinv}}{=} \alpha\left(\varepsilon\right)\cdot b_1\left(b_2 \rightharpoonup \chi(f)\right)(1_B) \\
& = \alpha\left(\varepsilon\right)\cdot b_1f(b_2 \oslash 1_B) = \alpha\left(\varepsilon\right)\cdot \left(b \triangleleft f\right)
\end{align*}
where in $(*)$ we used the fact that $\alpha \in \mathcal{R}(M)$.
That is, $\alpha(\varepsilon) \in M^{\pred{B}}$.

Conversely, if $m \in M^{\prescript{\star}{}{B}}$, then $\alpha_m \colon B^* \to M, \varphi \mapsto m\varphi(1_B)$, satisfies for all $f \in \pred{B}$ and $b\in B,$
\begin{align*}
\alpha_m(\chi(f)_\psi * \varphi)\cdot b^\psi & = m\left((b_2 \rightharpoonup \chi(f))* \varphi\right)(1_B)\cdot b_1 = m \varphi(1_B) \cdot b_1f(b_2 \oslash 1_B) = m \cdot (b \triangleleft f) \varphi(1_B) \\
& \stackrel{(\star)}{=} ((m \cdot b) \cdot f) \varphi(1_B) = m\varphi(1_B) \cdot (b \ltimes f) = \alpha_m(\varphi)\cdot (b \ltimes f)
\end{align*}
where in $(\star)$ we used the fact that $m \in M^{\prescript{\star}{}{B}}$. This means that $\alpha_m\in \mathcal{R}(M)$ by definition and in view of \cref{rem:predBlinear}. Moreover
\[(\varphi \rightharpoonup \alpha_m)(\psi) = \alpha_m(\psi * \varphi) = m\psi(1)\varphi(1) = \alpha_m(\psi)\varphi(1)\]
for all $\varphi,\psi \in B^*$.
That is, $\varphi \rightharpoonup \alpha_m = \alpha_m \varphi(1_B)$ for every $\varphi \in B^*$ and hence $\alpha_m \in {\mathcal{R}(M)^\rat}^{\mathrm{co}B}$. It follows that ${\mathcal{R}(M)^\rat}^{\mathrm{co}B} \cong M^{\prescript{\star}{}{B}}$.
\end{proof}


\subsection{A word on Tambara's construction and the Heisenberg double}

Let $B$ be an algebra and $A$ be a left $B$-module algebra. Then $\psi_{B,A}:B\otimes A\to A\otimes B,\,b\otimes a\mapsto b_1\cdot a\otimes b_2$
satisfies the conditions in \eqref{eq:tambara} and we have ${}_{A^\op}\mathfrak{M}_B(\psi_{B,A})\cong \mathfrak{M}_{A \#_{\psi_{B,A}} B}$.
Since $A^\op$ becomes a left $B^\cop$-module algebra, we can also consider $\psi_{B^\op,A^\cop}:B^\cop\otimes A^\op\to A^\op\otimes B^\cop$. This is $\psi$ from \cref{def:psi}
and hence ${}_{A}\mathfrak{M}_{B^\cop}(\psi)\cong \mathfrak{M}_{A^\op \#_{\psi} B^\cop}$.
As a particular case, since we noticed that $B^*$ is a left $B$ module, we can consider $A=B^*$ obtaining $\psi$ of \eqref{def:psi2} and ${}_{B^*}\mathfrak{M}_{B^\cop}(\psi)\cong \mathfrak{M}_{B^{*\op} \#_{\psi} B^\cop}\cong {}
_{(B^{*\op} \#_{\psi} B^\cop)^\op}\mathfrak{M}$.
Note also that, from \cite[Proposition 22]{CMZ},
the algebra $(B^{*\op} \#_{\psi} B^\cop)^\op$ is isomorphic to $B^{\op,\cop} \#_{\tau\psi\tau} B^*$ where $\tau$ is the flip of vector spaces.
Our next  aim is to show that the latter algebra is related with the notion of Heisenberg double. Recall from \cite[\S 6.4]{CMZ} that, given a bialgebra $L$, its Heisenberg double $\mathcal{H}(L)$ is the smash product $L\# L^*$ with multiplication defined by
\[(x \,\#\, \varphi)(y \,\#\, \phi) = {x}_2 \cdot y \,\#\, \varphi * ({x}_1 \rightharpoonup \phi). \]
In particular, for $L=B^{\op,\cop}, x,y \in B$ and $\varphi,\phi \in B^*$ we obtain
\begin{align*}
(x \,\#\, \varphi)(y \,\#\, \phi) &= {x^\cop}_2 \cdot^\op y \,\#\, \varphi *^{\op,\cop} ({x^\cop}_1 \rightharpoonup \phi) = {x_1}^\cop\cdot^\op y \,\#\, \varphi *^{\op,\cop} ({x_2}^\cop \rightharpoonup \phi) \\
&=  y\cdot {x_1}^\cop \,\#\, ({x_2}^{\cop} \rightharpoonup \phi)*\varphi =(y \,\#_{\tau\psi\tau}\, \phi)\cdot (x \,\#_{\tau\psi\tau}\, \varphi) = (x \,\#_{\tau\psi\tau}\, \varphi)\cdot^\op (y \,\#_{\tau\psi\tau}\, \phi)
\end{align*}
so that $B^{*\op} \#_{\psi} B^\cop \cong (B^{\op,\cop} \#_{\tau\psi\tau} B^*)^{\op} = \mathcal{H}(B^{\op,\cop})$ and hence ${}_{B^*}\mathfrak{M}_{B^\cop}(\psi)\cong \mathfrak{M}{}_{\mathcal{H}(B^{\op,\cop})}$.
Note that ${}_{B^*}\mathfrak{M}_{B^\cop}(\psi)$ is just ${}_{B^*}\mathfrak{M}_{B}(\psi)$ as considered in \eqref{eq:leftadj2} as $B^\cop$ and $B$ have the same underlying algebra.

\bibliography{references}
\bibliographystyle{acm}

\end{document}